\documentclass[journal,twoside,web]{ieeecolor}
\usepackage{generic}
\usepackage[mathscr]{euscript}
\usepackage[justification=centering]{caption}
\usepackage{amsmath,amssymb,amsfonts}
\usepackage[linesnumbered,ruled,vlined]{algorithm2e}
\SetKwInput{KwInput}{Input}                
\SetKwInput{KwOutput}{Output}  
\SetKwInput{kwProcedure}{Procedure}           
\usepackage{graphicx}
\usepackage{microtype}
\usepackage{multicol}
\usepackage[mathscr]{euscript}
\usepackage{bigints}
\usepackage{amssymb}
\makeatletter
\renewcommand*\env@matrix[1][\arraystretch]{%
  \edef\arraystretch{#1}%
  \hskip -\arraycolsep
  \let\@ifnextchar\new@ifnextchar
  \array{*\c@MaxMatrixCols c}}
\makeatother

\usepackage{bbm}

\let\euscr\mathscr \let\mathscr\relax
\usepackage[scr]{rsfso}
\usepackage{dirtytalk}
\usepackage{ntheorem}
\usepackage{hyperref}
\usepackage{outlines}
\usepackage{mathtools}
\usepackage{color}
\usepackage{float}
\usepackage{diagbox}
\usepackage[utf8]{inputenc}
\usepackage{lipsum}
\usepackage{graphicx}
\usepackage{subcaption}
\usepackage{cuted}
\usepackage{verbatim}
\newtheorem*{examp*}{Example}
\usepackage{marvosym}
\usepackage{array}
\usepackage{amsfonts}
\usepackage{enumerate}
\usepackage{graphicx}

\providecommand{\keywords}[1]{\textbf{\textit{Index terms---}} #1}

\newtheorem*{prob*}{Problem}

\newtheorem{example}{Example}
\newtheorem*{solution*}{Solution}
\newtheorem*{sol*}{Solution}
\newtheorem*{GOCP*}{GOCP}
\newtheorem*{EOCP}{ESIHLQR}

\newtheorem{theorem}{Theorem}
\newtheorem{problem}[theorem]{Problem}
\newtheorem{corollary}[theorem]{Corollary}
\newtheorem{lemma}[theorem]{Lemma}
\newtheorem{remark}[theorem]{Remark}
\newtheorem{cop*}{Convex Optimization Problem}
\newtheorem*{ihlqr*}{IHLQR}
\newtheorem{definition}[theorem]{Definition}
\newtheorem{proposition}[theorem]{Proposition}
\usepackage{textcomp}

\begin{document}
\title{
Suboptimal Consensus Protocol Design for a Class of Multiagent Systems 
}
\author{Avinash Kumar and Tushar Jain 
\IEEEmembership{Senior Member, IEEE  } 
\thanks{
Avinash Kumar and Tushar Jain are with Indian Institute of Technology Mandi, School of Computing and Electrical Engineering, Kamand, Himachal Pradesh - 175075, India (e-mail: d16005@students.iitmandi.ac.in, tushar@iitmandi.ac.in).}}
\maketitle
\begin{abstract}
This article presents a new technique for suboptimal consensus protocol design for a class of multiagent systems. The technique is based upon the extension of newly developed sufficient conditions for suboptimal linear-quadratic optimal control design, which are derived in this paper by an explication of a noniterative solution technique of the infinite-horizon linear quadratic regulation problem in the Krotov framework. 
%
For suboptimal consensus protocol design, the structural requirements on the overall feedback gain matrix, which are inherently imposed by agents dynamics and their interaction topology, are recast on a specific matrix introduced in a suitably formulated convex optimization problem. As a result, preassigning the identical feedback gain matrices to a network of homogeneous agents, which acts 
on the relative state variables with respect to their neighbors is not required. The suboptimality of the computed control laws 
is quantified by implicitly deriving an 
upper bound on the cost in terms of the solution of 
a convex optimization problem and initial conditions instead of specifying it \textit{a priori}.
Numerical examples are provided to demonstrate the implementation of  proposed approaches and their comparison with existing methods in the literature.
\end{abstract}


\section{Introduction}
\label{sec1_intro}

\IEEEPARstart{R}{ecently}, control design for multiagent systems has gained utmost interest amongst theorists and practitioners in different forms of achieving the application-dependent control objectives. One of the most important problems in coordination control of multiagent systems is the consensus protocol design problem, also known as the agreement problem. This problem has been widely addressed in the literature see, for instance, \cite{ren2005survey,borrelli2008dist, mesbahi2010graph, cao2010optimal, zhang2015design}. The reason for this exploration is that this problem serves as the background for developing control algorithms for more involved problems like formation control, synchronization, vehicle platooning, and flocking. The consensus problem requires the agents to appear at a common value (agreement) via exchange of relative and/or local information. In an optimal consensus design problem, the agents require to achieve agreement while optimizing a given objective function. Since only the relative and/or local information is shared among the agents, the control input becomes restrictive due to additional  
structural requirements on the overall feedback gain matrix dictated by the topology of agent interaction and individual agent dynamics. 
These structural conditions render the underlying control optimization problem as nonconvex. It is, therefore difficult, if not impossible, to find a closed-form solution for the optimal controller, or it may not even exist \cite{jiao2019distributed, jiao2019suboptimality}. 
The generic optimal consensus protocol design (GOCPD) problem is as stated below for linear agents with a quadratic cost.

\begin{problem}[GOCPD]\label{sec2_ocnp}
Consider a group of $N$ agents with the individual dynamics
$
\mathbf{\dot{x}}_i= A_i \mathbf{x}_i + B_i \mathbf{u}_i, i = 1,2, \hdots, N,
$
$\mathbf{x}_i(t_0)=\mathbf{x}_i{_0}, \mathbf{x}_i \in \mathbb{R}^{n}$, $\mathbf{u}_i \in \mathbb{R}^{m}$ with the pair $(A_i, B_i)$ denoting their distribution matrix pair of appropriate dimensions, which
communicates over a given bidirectional network topology. Let $\mathcal{N}_i$ denote the neighborhood of agent $i$. 
Compute a distributed diffusive control law of the form
\begin{equation}\label{sec2_ocn_ip}
\mathbf{u}_i= f_i (\mathbf{x}_i- \mathbf{x}_j) ,
\end{equation}
where $j \in \mathcal{N}_i$, and $f_i(\bullet)$ denotes a vector-valued linear function of relative states,
such that the cost 
\begin{multline}\label{sec2_ocn_cost}
\hat{J}= \bigintss_{t_0}^{\infty}\left( \frac{1}{2}\sum_{i =1}^{N}\sum_{j \in \mathcal{N}_i} \left( \mathbf{x}_i-\mathbf{x}_j\right) ^T Q_{ij} \left( \mathbf{x}_i-\mathbf{x}_j \right)\right.\\ 
 + \left.\sum_{i=1}^N  \mathbf{u}_i^T R_i \mathbf{u}_i \right) dt, 
 \end{multline}%
where $Q_{ij} = Q_{ij}^T = \underline{Q} \succeq \mathbf{0}, \forall  i,j, R_i =R_i^T=\underline{R} \succ \mathbf{0}, \forall  i$
is minimized, and the consensus is achieved, i.e., $$ \lim_{t \rightarrow \infty} ||\mathbf{x}_i(t)-\mathbf{x}_j(t)|| \rightarrow {0}, \forall i,j \in \{1,2, \ldots, N\} \text{ with } i \neq j.$$ 
\end{problem}
Due to the inherent nature of the required control law \eqref{sec2_ocn_ip}, Problem \ref{sec2_ocnp} is not tractable, except for the case of single integrator agents communicating over a complete graph topology for which the closed-form expression for the optimal controller can be computed \cite{cao2010optimal, kumar2020alternative}.
A suboptimal version of Problem \ref{sec2_ocnp} is usually tackled in the literature where consensus protocols are computed so that the cost $\hat{J}$ is bounded and for an \textit{a priori} given parameter $\gamma>0$ satisfies $\hat{J}< \gamma$. The techniques employed in the literature are based upon the suboptimal linear-quadratic control design \cite{jiao2019suboptimality, nguyen2015sub, zeng2017structured}. Typically,  an identical feedback gain matrix for a network of homogeneous agents is computed, and subsequently, the control law \eqref{sec2_ocn_ip} is determined using its Kronecker product with the Laplacian matrix satisfying the network-imposed structural requirement. 
In \cite{jiao2019suboptimality}, this methodology is used to develop sufficient conditions for homogeneous multiagent systems. In \cite{borrelli2008dist}, the relationship between the stability of a large-scale system and the sparsity pattern of the feedback matrix for the considered problem of identical agents is studies, while the identical local feedback matrices are computed using the LQR-based approach for solving the GOCPD tackling the problem.
Numerous other approaches have been proposed in the literature to address the consensus problem. One of the approaches is to decompose the overall control signal in local and global control signals. In \cite{nguyen2015sub}, a hierarchical LQR based approach is developed to design a consensus protocol with two terms- local and global. In \cite{qiu2016distributed}, this technique was used to design a control input containing three terms: local information, local projection, and local sub-gradient with the assumption of a uniformly jointly connected communication network and bounded time-varying edge weights. An LMI based approach was proposed in \cite{semsar2009optimal} after proving that solution of the Riccati equation cannot be used to compute a solution that achieves consensus based upon the partition of the state space into consensus space and orthonormal subspace. In \cite{xie2019global}, global consensus protocols are designed where the agents are integrators and have individual objective functions known only to themselves under the assumption of strongly connected topology. In \cite{wang2021accelerated}, an iterative accelerated optimal consensus algorithm is developed by decomposing the problem into two independent sub-problems. An iterative procedure for computing the distributed suboptimal controller for discrete-time multiagent systems with the subsystems having identical decoupled linear dynamics is also developed in \cite{zhang2015design}. Using the receding-horizon control technique, the suboptimal controllers for second-order linear and nonlinear multiagent systems are computed in \cite{zhang2017predictive}. With a certain level of $\mathcal{H}_2$ or $\mathcal{H}_\infty$ performance, the consensus protocols have also been synthesized in the literature, for example, in \cite{li2011h, deshpande2012sub, jiao2021h2}.

As seen above that majorly the suboptimal consensus protocol design problem for a multiagent system is addressed by solving the suboptimal LQR problem for a single agent, which implicitly requires solving the algebraic Riccati equation either using some iterative procedure or noniteratively. The  standard  results  on  suboptimal  control  design  for  linear  systems  with  the quadratic cost can be found in \cite{trentelman2012control} and \cite{iwasaki1994linear}. In the noniterative procedure, it also requires an \textit{a priori} knowledge of the upper bound $\gamma$, which is not trivial to specify given that the optimal cost is not known. Moreover, for a network of homogeneous agents, the overall feedback matrix is computed as the Kronecker product of the Laplacian matrix and an identical feedback gain for all agents. Since the optimal cost is unknown at the outset and under the given initial condition of agents, the optimal control problem  may become infeasible. Also, specifying the identical feedback gain for all agents may negatively affect the degree of suboptimality.

In this work, firstly, we formulate a \textit{joint problem}, which is defined as to compute the suboptimal consensus protocol for a class of multiagent systems and the upper bound $\gamma$ on the cost $\hat{J}$ \eqref{sec2_ocn_cost}, where the state dynamics of each agent differs in the sense that the state matrix of all agents is identical, but the input distribution matrix. For the case of the homogeneous agents, the mandatory restriction of identical feedback gain for all agents is also removed. The solution to the above problem is derived from the solution of the second joint problem, introduced for a single agent, where the objective is to compute the $\epsilon-$suboptimal control law and the parameter $\epsilon$. As a consequence of the solution to the first problem, we also provide a solution to the problem of determining an upper bound on the cost for any given consensus protocol. 
The solution to aforesaid problems is based upon the explication of a noniterative solution technique for optimal control problems  
in the rather less-explored Krotov framework. Within this framework, the paradigm of a noniterative solution technique for the LQR, linear-quadratic tracking, and bilinear optimal control problems is firstly developed in our previous work \cite{kumar2019some}. In this framework, the optimal control problem is translated into another equivalent optimization problem by a selection of a Krotov function. This equivalent optimization problem is then solved using the so-called Krotov method \cite{Kro95}, which is an iterative algorithm and may not be a practical control approach considering the real-time implementation constraints. 
Nevertheless, the selection of the Krotov function affects the above-mentioned iterative procedure.  
In this work, we propose a new method to obtain a direct
(or noniterative) \textit{suboptimal} solution of the resulting equivalent
optimization problem for multiagent and single-agent systems using Krotov sufficient conditions for global optimality. 
Furthermore, for both the problems, the upper bound on the cost is also computed in comparison to the existing results in the literature, where it is specified \textit{a priori}. 
The contributions of this paper are many folds:
\begin{enumerate}[(i)]
\item New sufficient conditions are derived for obtaining the control laws for the $\epsilon-$suboptimal LQR problem in the Krotov framework. The parameter $\epsilon$ is also computed.
\item A new method is derived to obtain the suboptimal consensus protocol for a class of multiagent systems with linear dynamics in the Krotov framework. The upper bound on the cost functional is also computed implicitly instead of specifying it \textit{a priori}.
\item An upper bound on the cost functional is explicitly computed for any consensus protocol of a multiagent system. This contribution shall be useful for control practitioners.
\item A comparison of proposed methods with state-of-the-art methods is also presented through numerical simulations. 
\end{enumerate}

The rest of the paper is organized as follows. In section \ref{sec2_pmb}, the mathematical background of the Krotov framework is briefly discussed along with the considered problems in this work. Section \ref{sec3_olmi_slqr} presents the new results for suboptimal linear-quadratic control design by proposing suitable Krotov functions. Section \ref{sec4_scpd} presents the new method to compute the 
suboptimal consensus protocol for a class of multiagent systems. In section \ref{sec5_ne}, numerical results are presented. With respect to the existing results in the literature, a comparative analysis is also presented in this section, followed by an application of the proposed method to a practical problem of roller consensus in a paper drying machine with a nonidentical input matrix for all agents. Finally, the concluding remarks and future work  are discussed in section \ref{sec6_conc}.

\section{Preliminaries and Problem formulation}
\label{sec2_pmb}
This section presents a brief overview and important results within the Krotov framework, which are required to solve the problems formulated in the subsequent subsections.
\subsection{Krotov Framework}
\label{sec2_pmb_sub_kf}
The Krotov framework is  based upon the application of the extension principle to optimal control problems. 
\subsubsection{Extension Principle}
\label{sec2_pmb_sub_kf_sub_ep}
The idea of the extension principle is to reformulate the constrained optimization problem as an unconstrained problem with (possibly) a bigger solution space in such a way that the
solution to the latter problem is the same as that of the former problem 
\cite{Kro95,gurman2016certain}.
Consider  a scalar-valued functional $L(\mathbf{v})$ defined over a set $\mathbb{D}$ (i.e. $\mathbf{v} \in \mathbb{D}$), and the optimization problem to be solved is\\
\textit{(Original) Problem:}
 Find $\bar{\mathbf{v}}_d$ such that $d=\inf_{\mathbf{v} \in \mathbb{D}}L(\mathbf{v})$, where $d \triangleq L(\bar{\mathbf{v}}_d) $, and $L(\bullet): \mathbb{D} \rightarrow \mathbb{R}$.
  
 \noindent{Instead of directly solving this problem, another  optimization problem, which is equivalent to the original, is formulated by using the extension principle as} 
 
 \noindent \textit{(Equivalent) Problem :} 
 Find $\bar{\mathbf{v}}_{\tilde{d}}$ such that $\tilde{d}= \inf_{\mathbf{v} \in \tilde{\mathbb{D}}}\tilde{L}(\mathbf{v})$, where $\tilde{d} \triangleq \tilde{L}(\bar{\mathbf{v}}_{\tilde{d}}) $, $\tilde{L}(\bullet): \tilde{\mathbb{D}} \rightarrow \mathbb{R}$, $\tilde{\mathbb{D}} \supseteq \mathbb{D}$ and $\tilde{L}(\mathbf{v})=L (\mathbf{v}) \forall  \mathbf{v}  \in \mathbb{D}$. 

The essence of the extension principle is that the equivalent problem can be easier to solve than the original problem by selecting the non-unique functional $\tilde{L}$. However, its  selection 
and the characterization of the set $\tilde{\mathbb{D}}\supseteq \mathbb{D}$ remain open problems in the literature \cite{kumar2019some}. Applying the above extension principle to a generic optimal control problem, a specific functional is defined, and subsequently, the sufficient conditions of global optimality for solving that generic problem are provided next. 

 \subsubsection{Krotov Sufficient Conditions}
 \label{sec2_pmb_sub_kf_sub_ksc}
Consider the generic optimal control problem stated below.
\begin{problem}
\label{sec2_prob_ocp}
Compute an optimal feedback control law $\mathbf{u}^*(t)$ which minimizes the performance index
\begin{equation}  \label{sec2_gocp_cost}
\begin{gathered}
J(\mathbf{x}(t),\mathbf{u}(t))= l_f(\mathbf{x}(t_f)) + \int_{t_0}^{ t_f} l(\mathbf{x}(t),\mathbf{u}(t),t) dt \\
\text{s.t. } \dot{\mathbf{x}}(t) = f(\mathbf{x}(t),\mathbf{u}(t),t) \ \text{ with } {\mathbf{x}}(t_0) =\mathbf{x}_0, 
\end{gathered}
\end{equation}
where $l(\mathbf{x}(t),\mathbf{u}(t),t)$ is the running cost, $l_f(\mathbf{x}(t_f))$ is the terminal cost, $\mathbf{x}(t) \in  \mathbb{X} \subset \mathbb{R}^n$ is the state vector and $\mathbf{u}(t) \in \mathbb{U} \subset \mathbb{R}^m$ is the control input vector. Also, $l_f : \mathbb{R}^n \rightarrow \mathbb{R} $ and $l : \mathbb{R}^n \times \mathbb{R}^m \times [t_0, t_f] \rightarrow \mathbb{R}$ are continuous, and $\mathbf{x}(t_f) \in \mathbb{X}_f$ with $\mathbb{X}_f$ denote the terminal set. 
\end{problem}
\begin{theorem}[Krotov Theorem] \label{sec2_thm1} 
Let $q(\mathbf{x}(t),t)$ be a piecewise smooth function, denoted as Krotov function. Then, \eqref{sec2_gocp_cost} can be equivalently expressed as 
\begin{equation*}
J_{eq}(\mathbf{x},\mathbf{u};q)= s_f(\mathbf{x}(t_f)) +q(\mathbf{x}_0,t_0)
+ \int_{t_0}^{ t_f} s(\mathbf{x}(t),\mathbf{u}(t),t;q) dt 
\end{equation*}
where
\begin{align*}
s(\mathbf{x},\mathbf{u}(t),t;q) & \triangleq  \frac{\partial q}{\partial t} + l(\mathbf{x}(t),\mathbf{u} (t),t)+\frac{\partial q}{\partial \mathbf{x}} {f} (\mathbf{x}(t), \mathbf{u}(t),t), \\
s_f(\mathbf{x}(t_f);q) &\triangleq l_f(\mathbf{x}(t_f)) -q(\mathbf{x}(t_f),t_f).
\end{align*}
\end{theorem}
Since the function $q$ is non-unique, the representations given in Theorem \ref{sec2_thm1} are also not unique, see \cite[Remark 1]{kumar2019some}.
\begin{definition}[Admissible process]
A process (or pair) $(\mathbf{x}(t), \mathbf{u}(t))$ is admissible whenever it satisfies the dynamical equation $\dot{\mathbf{x}}= f (\mathbf{x},\mathbf{u},t)$ and the state and input constraints.
\end{definition}
\begin{theorem}[Krotov Sufficient Conditions] \label{sec2_thm2} 
If $\left({\mathbf{x}^*(t)},{\mathbf{u}^*(t)}\right)$ is an \textit{admissible process}  such that
\begin{align*}
s(\mathbf{x}^*,\mathbf{u}^*,t;q) &=   \min_{\mathbf{x} \in \mathbb{X}, \mathbf{u} \in \mathbb{U} }  s(\mathbf{x},\mathbf{u},t;q),  \forall t  \in   [t_0,t_f),\\
s_f(\mathbf{x}^*(t_f);q) &= \min_{\mathbf{x} \in \mathbb{X}_f}  s_f(\mathbf{x};q)
\end{align*}
then $(\mathbf{x}^*(t), \mathbf{u}^*(t))$ is an optimal process.
\end{theorem}
The proof of Theorems \ref{sec2_thm1} and \ref{sec2_thm2} can be found in \cite[Section 2.3]{Kro95}. 
In the literature, the optimal process for linear and nonlinear systems is computed using the so-called Krotov iterative methods, see, e.g., \cite{maximov2010smoothing, Raf18}, yielding a sequence of improving admissible processes because the equivalent optimization problem may be nonconvex, see \cite{kumar2019some} for finite-time LQR problem. Moreover, although it is well-known that the global optimal consensus protocol may not exist for multiagent systems, the Krotov optimal control framework has not been explored for computing even the suboptimal solutions to the best of our knowledge. In this paper, the main objective is to obtain the suboptimal control solutions using Krotov sufficient conditions and via a suitable selection of the Krotov functions for single-agent and multiagent systems.
\begin{definition}[Solving function]
Given a Krotov function $q(\mathbf{x}, t)$, whenever an optimal or suboptimal process is directly (or noniteratively) obtained, then
$q(\mathbf{x}, t)$ is defined as the solving function.
\end{definition}

\subsection{Problem Formulation}
In this paper, the following suboptimal consensus protocol design problem constituting of $N$ agents communicating over a given bidirectional topology is considered.
\begin{problem}[$\gamma-$Suboptimal Consensus Protocol Design]
\label{sec2_scdp}
Given a group of agents with state dynamics of the $i-$th agent as
\begin{equation}\label{eq:agentdynamicsdiffB}
{\dot{\mathbf{x}}}_i= A \mathbf{x}_i+ B_i \mathbf{u}_i , \ i \in \{1,2, \hdots, N\}, \mathbf{x}_i(t_0)=\mathbf{x}_{i0},
\end{equation}
where $\mathbf{x}_i\in\mathbb{R}^n, \mathbf{u}_i\in\mathbb{R}^m$.
Design a $\gamma-$suboptimal consensus control law of the form 
$
\mathbf{u}_i= f_i (\mathbf{x}_{i}-\mathbf{x}_{j}), \forall j\in\mathcal{N}_i. 
$
Moreover, compute the upper bound $\gamma>0$ on the cost functional $\hat{J}$ in \eqref{sec2_ocn_cost}. 
\end{problem}
Numerous works in the literature, typically, consider identical dynamics for all the agents \cite{borrelli2008dist, cao2010optimal,gupta2005sub,jiao2018suboptimality}, unlike the dynamics considered in Problem \ref{sec2_scdp}, which is practically more relevant than the former. For example, any two DC motors of identical ratings may have the same time constant in the first-order transfer function but different rotor/stator gains \cite{golnaraghi2017automatic}. Secondly, for $\gamma-$suboptimal control design procedures in the literature, the information of $\gamma$ is assumed to be known \textit{a priori}. 
In this work, these assumptions are not considered at the outset. In fact, distinct feedback gain for all agents offers inherent flexibility to the degree of suboptimality in the sense that the computed cost may be less than the cost computed considering the identical feedback gain for all agents. Moreover, prespecifying injudiciously the upper bound on the cost functional may yield no solution to the suboptimal control problem for some given initial conditions of the state variables.  
Otherwise, in order to obtain the solution for a given upper bound, it may restrict the initial conditions of agents since in the existing literature, once such a feedback gain is computed ensuring the consensus, there lies a strong dependency between the \textit{a priori} given upper bound and a ball containing the initial conditions of agents. In this work, the former benefit of not adhering to identical feedback gains for a network of homogeneous agents  and the latter issue of strong dependency are numerically demonstrated 
in the simulation section with regards to the existing state-of-the-art approaches.


The solution to the above problem shall be computed by 
reformulating the overall problem from an LQR perspective, where the cost function to be minimized is reformulated in terms of the error dynamics between the states of interconnected agents 
subject to the new admissible process. 
In this context, to investigate the suboptimal solutions for a multiagent system, it requires firstly to determine the suboptimal solutions to the classical infinite-horizon (IH) LQR control problem for linear systems or single-agent systems. Several results on the latter problem are broadly gathered in \cite{trentelman2012control,kirk2004optimal}. In this paper, new conditions are derived for computing the suboptimal solutions to the IHLQR problem using the Krotov sufficient conditions, which is another significant contribution of this work. These conditions shall then be used to address Problem \ref{sec2_scdp}.
\begin{definition}[$\epsilon- $suboptimality]
For the IHLQR problem, an admissible process  $(\mathbf{x}(t),\mathbf{u}(t))$ is defined as an $\epsilon$-suboptimal process and subsequently, a control law of the form
\begin{align}\label{eq:controllaw}
\mathbf{u}= -\frac{1}{2} R^{-1}B^T \left(P+P^T \right) \mathbf{x} 
\end{align}
is called $\epsilon$-suboptimal whenever $P \in \mathcal{P}$ where 
$
    \mathcal{P} = \left\lbrace P |J(\mathbf{u})< J^*+\epsilon\right\rbrace,
    $
with 
\begin{align}\label{eq:ihlqrcost}
     J(\mathbf{u})= \int_{t_0}^{\infty} \left(\mathbf{x}^TQ \mathbf{x}+\mathbf{u}^T R \mathbf{u}  \right) dt,
\end{align}
and $J^*$ denoting the optimal cost.
\end{definition}
\begin{problem}[$\epsilon$-suboptimal LQR Design]
\label{sec2_p3}
Consider the system 
\begin{align}\label{eq:singleagent}
    \mathbf{\dot{x}}=A \mathbf{x}+B \mathbf{u}, \mathbf{x}(t_0)=\mathbf{x}_0
\end{align}
    with cost functional \eqref{eq:ihlqrcost}. Compute 
\begin{enumerate}
    \item an $\epsilon-$suboptimal control law \eqref{eq:controllaw}, and 
    \item the parameter $\epsilon$.
\end{enumerate}
\end{problem}
Note that in the literature, the solution to the Problem \ref{sec2_p3} is majorly presented by prespecifying the upper bound on the cost, which may be infeasible whenever the upper bound is injudiciously selected. Nevertheless, the optimal cost is typically known in the Problem \ref{sec2_p3} unlike in the Problem \ref{sec2_scdp}.  In this work, we explicitly characterize this upper bound while computing the suboptimal controller.



\section{Suboptimal LQR Design}
\label{sec3_olmi_slqr}
Using Krotov sufficient conditions, 
solving functions are proposed in \cite{kumar2019some} to compute the optimal process for the finite- and infinite-horizon LQR problem. In this section, we shall compute the new suboptimal control law and an upper bound on the cost functional for the IHLQR problem. Prior to that, we will derive the results for computing the $\epsilon-$suboptimal process, given the parameter $\epsilon>0$.
The equivalent optimization problem of the suboptimal IHLQR (SIHLQR) problem, as per Theorems \ref{sec2_thm1} and \ref{sec2_thm2}, is stated as 
\begin{EOCP}[Equivalent SIHLQR problem] 
\label{sec3_eihlqr}
Given $\epsilon>0$. Find an $\epsilon-$suboptimal process 
which
\begin{enumerate}
    \item $\min\limits_{\mathbf{(x,u)} \in \mathbb{R}^n\times  \mathbb{R}^m }  s(\mathbf{x}(t),\mathbf{u}(t),t;q)$, where
    \begin{equation}\label{eq:funcs}
s=\frac{\partial q}{\partial t}+\dfrac{\partial q}{\partial \mathbf{x}} \left(A\mathbf{x}+B \mathbf{u} \right) +  \left(\mathbf{x}^T Q \mathbf{x} + \mathbf{u}^T R \mathbf{u} \right),
\end{equation}
\item $\min\limits_{\lim\limits_{t_f \rightarrow \infty}\mathbf{x}(t_f)} s_f(\mathbf{x}(t_f))$, where $s_f = -q(\mathbf{x}(t_f),t_f)$.
\end{enumerate}
\end{EOCP}
\begin{proposition}\label{prop:convexs}
For the ESIHLQR, let the Krotov function be chosen as $q(\mathbf{x},t)=\mathbf{x}^TP\mathbf{x}$. The function $s(\mathbf{x},\mathbf{u},t)$ is a strictly convex function in $(\mathbf{x},\mathbf{u})$ if and only if the matrix $P$ satisfies
\begin{align}\label{sec3_l1_p_lmi}
\begin{bmatrix}
 \Gamma(P) &  \left(P +P^T\right)B \\
 B^T\left( P+P^T\right)  & 4R 
\end{bmatrix}   \succ \mathbf{0},
\end{align}
where $\Gamma(P)= A^TP+PA+Q$.
\end{proposition}
\begin{proof}
By instantiation of the results in \cite[Section II.B]{kumar2019some} for infinite-horizon.
\end{proof}
\begin{theorem}
Let $\Upsilon(\Psi,\Phi)=(\Psi+\Psi^T ) \Xi ( \Phi+\Phi^T )$, where $\Xi=\Xi^T\succeq \mathbf{0}$ is symmetric. The following statements are true.
\begin{enumerate}
    \item $\Upsilon(\Psi,\Phi)^T = \Upsilon(\Phi,\Psi)$.
     \item In addition, if $\Psi = \Phi$, then $\Upsilon(\Phi,\Phi)\succeq \mathbf{0}$. 
\end{enumerate}
\end{theorem}
\begin{proof}
$\Upsilon(\Psi,\Phi)^T=\left((\Psi+\Psi^T ) \Xi ( \Phi+\Phi^T ) \right)^T= (\Phi+\Phi^T) \Xi ( \Psi+ \Psi^T )= \Upsilon(\Phi, \Psi)$. This completes the proof of $1)$.

With $\Psi= \Phi$, $\Upsilon(\Phi,\Phi)=(\Phi+\Phi^T) \Xi (\Phi+ \Phi^T)$ and since $\Xi \succeq \mathbf{0}$, $\Upsilon \succeq \mathbf{0}$.
\end{proof}   
\begin{proposition}
The function $q=\mathbf{x}^TP\mathbf{x}$ is a solving function for ESIHLQR. 
Moreover, the closed-loop system with the control law \eqref{eq:controllaw}
is stable and the law is $\epsilon-$suboptimal if the following conditions are satisfied for some $P$
\begin{enumerate}
    \item $s(\mathbf{x},\mathbf{u},t;q)$ in \eqref{eq:funcs} is a strictly convex function in $(\mathbf{x},\mathbf{u})$,
\item $A-\frac{1}{2}\Upsilon(I,P)
$ is Hurwitz, where 
$\Xi = BR^{-1}B^T$
with $I$ denoting the identity matrix of appropriate dimension,
    \item $\mathbf{x}_0^T (P+\check{P}) \mathbf{x}_0<J^*+\epsilon$, where $\int_{t_0}^{\infty}sdt = \mathbf{x}_0^T\check{P}\mathbf{x}_0$.
\end{enumerate}
\end{proposition}
\begin{proof}
From Proposition \ref{prop:convexs}, the strict convexity of the functional $s(\bullet)$ is equivalent to that the inequality \eqref{sec3_l1_p_lmi} holds for some matrix $P$. If there also exists a stabilizing matrix $P$ satisfying the inequality \eqref{sec3_l1_p_lmi} such that the closed-loop state matrix $A-\frac{1}{2}\Upsilon(I,P)
$ is Hurwitz, then the cost is finite. Note that 
the cost is optimal if $P$ is the solution of $s=0$ and satisfies $(P+P^T)\succ \mathbf{0}$ \cite[Corollary 4]{kumar2019some}. Using Theorem \ref{sec2_thm1}, the equivalent cost function for \eqref{eq:ihlqrcost}-\eqref{eq:singleagent} after substituting the control law \eqref{eq:controllaw}
reads as
$
J_{eq} =\mathbf{x}_0^T P \mathbf{x}_0 + \int_{{t_0}}^\infty s(\mathbf{x},t;q) dt,
$
where 
    $s=\mathbf{x}^T\left[\Gamma(P)-\frac{1}{4}\Upsilon(P,P)
    \right]\mathbf{x}$.
Solving the above integral, it yields $J_{eq} = \mathbf{x}_0^T(P+\check{P})\mathbf{x}_0$. Since the cost is bounded above, if the matrix $P$ satisfies $\mathbf{x}_0^T(P+\check{P})\mathbf{x}_0<J^*+\epsilon$ for some given $\epsilon>0$, then the control law is $\epsilon-$suboptimal. 
\end{proof}
The above result raises a question: how to \textit{a priori} specify the upper bound on the cost function? Several works have been reported in the literature for computing a suboptimal controller for the LQR problem given an upper bound on the cost \cite{iwasaki1994linear,kleinman1968design, johansen2002explicit}. For suboptimal LQR problems, the upper bound can be \textit{safely} prespecified because the optimal cost value is known. If this information is not available, then the selected upper bound may yield an infeasible solution. Subsequently, another question arises even when the optimal cost is known, i.e., for some given arbitrary initial condition, what is the relationship between $\epsilon$ and $P$ to ensure a feasible solution? This requires imposing some additional conditions after the computation of feedback gain. The first two conditions in the above proposition ensure that the cost is bounded above, while the third condition merely ensures that the matrix $P$ needs to be computed for the \textit{a priori} given upper bound. Nevertheless, the computation of such a matrix $P$ might be challenging. Our following results focus on computing the upper bound directly once the boundedness of the cost is established.

In the next result, a new Krotov function is proposed by which $\epsilon-$suboptimality shall be later quantified.
\begin{proposition}
Consider the closed-loop system defined by \eqref{eq:controllaw} and \eqref{eq:singleagent} with the cost \eqref{eq:ihlqrcost}.
Let the Krotov function be chosen as $q=\mathbf{x}^T(P-\bar{P})\mathbf{x}$.
The function $s(\mathbf{x},\mathbf{u},t;q)$ is a strictly convex function in $(\mathbf{x},\mathbf{u})$
\begin{enumerate}
    \item \label{item:NandS} if and only if
     \begin{align}\label{eq:NandS}
      \begin{bmatrix}
  \Gamma(P-\bar{P}) + \frac{1}{2}\Upsilon(\bar{P},P) & (P+P^T )B\\
  B^T (P+P^T ) & 4R
  \end{bmatrix} 
  \succ \mathbf{0}
    \end{align}
    \item \label{item:suff} if 
    \begin{multline}
\left[
  \begin{matrix}\label{eq:suff}
    \Gamma(P-\bar{P})-Q-\frac{1}{2}\Upsilon(P,P)-\frac{1}{4}\Upsilon(\bar{P},\bar{P}) \\
    -B^T(\bar{P}+\bar{P}^T)
  \end{matrix}\right.                
\\
  \left.
  \begin{matrix}
  -(P+P^T)B \\
      -2R
  \end{matrix}\right] \succ \mathbf{0}
\end{multline}
\end{enumerate}
for some matrices $P$ and $\bar{P}$.
\end{proposition}
\begin{proof}
Firstly, we show the strict convexity of the functional $s(\bullet)$ is equivalent to \eqref{eq:NandS}. Subsequently, we show that \eqref{eq:suff} implies \eqref{eq:NandS}. Substituting $q=\mathbf{x}^T(P-\bar{P})\mathbf{x}$ and \eqref{eq:controllaw}
in \eqref{eq:funcs}, and simplifying it reads as
$
    s = \mathbf{x}^T \left( \Gamma(P-\bar{P}) -\frac{1}{4} \Upsilon(P,P)+\frac{1}{2} \Upsilon(\bar{P},P) \right) \mathbf{x},
    $
which is strictly convex iff the weight matrix of the quadratic term is positive definite, and by applying the Schur complement lemma can be expressed as the inequality \eqref{eq:NandS}. Now, adding the term $Q+\frac{1}{4}(\Upsilon(P,P) + \Upsilon(\bar{P},\bar{P}))$, which is positive semidefinite, to the left-hand side of the inequality \eqref{eq:suff} yields the inequality \eqref{eq:NandS}. The inequality still holds since a positive semidefinite term is added to the left-hand side.  This completes the proof.
\end{proof}
\begin{proposition}
If $P$ and $\bar{P}$ satisfy
\begin{align}\label{eq:decomposedpbar}
    \begin{bmatrix}
    -\Gamma(\bar{P}) & (P+P^T-\bar{P}-\bar{P}^T)B\\
    B^T(P+P^T-\bar{P}-\bar{P}^T) & 4R
    \end{bmatrix} \succ \mathbf{0},
\end{align}
and \eqref{sec3_l1_p_lmi} also holds,
then \eqref{eq:suff} can be equivalently expressed as \eqref{eq:bigLMI}. 

\begin{strip}
\rule{\textwidth}{1pt}
\begin{equation}\label{eq:bigLMI}
 \renewcommand\arraystretch{1.3}
\left[
	\begin{array}{cc|cc}
		\Gamma(P) &  \left(P +P^T\right)B & \mathbf{0} & \mathbf{0} \\
 B^T\left( P+P^T\right)  & 4R & \mathbf{0} & \mathbf{0}\\ \hline
		\mathbf{0} & \mathbf{0} & -\Gamma(\bar{P}) & (P+P^T-\bar{P}-\bar{P}^T)B\\
		\mathbf{0} & \mathbf{0} &
    B^T(P+P^T-\bar{P}-\bar{P}^T) & 4R
	\end{array}
	\right] \succ \mathbf{0}.
\end{equation}
\rule{\textwidth}{1pt}
\end{strip}
\end{proposition}
\begin{proof}
The block matrix in \eqref{eq:suff} can be readily obtained by adding the Schur complement of the first diagonal block matrix and the second diagonal matrix in \eqref{eq:bigLMI}. 
Since the inequality \eqref{sec3_l1_p_lmi} holds for some $P$, if the inequality \eqref{eq:decomposedpbar} holds for some $\bar{P}$, then \eqref{eq:suff} also holds.
\end{proof}
\begin{lemma}
\label{sec3_l3}
Let $\bar{A} \in \mathbb{R}^{n \times n}$ be a Hurwitz matrix. Then, for any matrix $W \in \mathbb{R}^{n \times n}$, there exists a unique solution $ P_w \in \mathbb{R}^{n \times n}$ of the equation
$
\bar{A}^TP_w+P_w\bar{A}+W= \mathbf{0},
$
given as
$ 
P_w= \int_{0}^{\infty} \left( e^{t \bar{A}^T}W e^{t\bar{A}}\right) dt.
$
Furthermore, if $W$ is positive (semi-)definite, then $P_w$ is positive (semi-)definite.
\end{lemma}
\begin{proof}
See \cite[Proposition 3.2]{terrell2009stability}.
\end{proof}
The next theorem and the subsequent lemma are crucial for stating one of the main results of this paper in proposition \ref{sec3_prop5}.
\begin{theorem}
\label{sec3_thm3}
Consider the stable system
$\dot{\mathbf{x}}(t)= \bar{A} \mathbf{x}(t),    \mathbf{x}(t_0)=\mathbf{x}_0$
and the quadratic performance index 
$J  =\int_{t_0}^{\infty} \left(  \mathbf{x}^T \bar{Q} \mathbf{x} \right)  dt$
where $\bar{Q} \succeq \mathbf{0}$ is the weighting matrix. 
Let the Krotov function be chosen as $q=\mathbf{x}^TY\mathbf{x}$. 
\begin{enumerate}
    \item For any matrix $Y\succeq \mathbf{0}$,
$J_{eq}(Y) \geq J^* = \mathbf{x}_0^T Y^* \mathbf{x}_0$, where $Y^*$ is the solution of $s=0$.
\item If $t_0 = 0$, then for any matrix $Y$,
$J_{eq}(Y) = J^*$.
\end{enumerate}
\end{theorem}
\begin{proof}
Substituting $q=\mathbf{x}^TY\mathbf{x}$ and the solution to the autonomous stable system in Theorem \ref{sec2_thm1}, we get
\begin{align}\label{eq:jeqauto}
    J_{eq}(Y) = \mathbf{x}_0^TY\mathbf{x}_0 + \int_{t_0}^\infty \mathbf{x}_0^T e^{\bar{A}^Tt}(Y\bar{A}+\bar{A}^TY+\bar{Q})e^{\bar{A}t}\mathbf{x}_0 dt.
\end{align}
The functional $s=0$ is equivalent to $\bar{A}^TY+Y\bar{A}=-\bar{Q}$. Using Lemma \ref{sec3_l3}, the solution to the latter equation is given as $Y^* = \int_{t_0}^\infty  e^{\bar{A}^Tt}\bar{Q}e^{\bar{A}t} dt$. Rearranging \eqref{eq:jeqauto}, it reads
\begin{align*}
    J_{eq}(Y) &= \mathbf{x}_0^T\left(Y + \int_{t_0}^\infty e^{\bar{A}^Tt}(Y\bar{A}+\bar{A}^TY)e^{\bar{A}t} dt + Y^*\right)\mathbf{x}_0\\
    &= \mathbf{x}_0^T\left(Y + \int_{t_0}^\infty \frac{d}{dt} \left(e^{\bar{A}^Tt}Ye^{\bar{A}t}\right) dt + Y^*\right)\mathbf{x}_0\\
    &= \mathbf{x}_0^T\left(Y  -e^{\bar{A}^Tt_0}Ye^{\bar{A}t_0} + Y^*\right)\mathbf{x}_0
\end{align*}
\begin{enumerate}
    \item Since $\bar{A}$ is Hurwitz, $\mathbf{x}_0^TY\mathbf{x}_0 \geq \mathbf{x}_0^Te^{\bar{A}^Tt_0}Ye^{\bar{A}t_0}\mathbf{x}_0 $ if $Y\succeq \mathbf{0}$. Hence, $J_{eq}(Y) \geq J^* = \mathbf{x}_0^T Y^* \mathbf{x}_0$.
    \item If $t_0=0$, then $\mathbf{x}_0^TY\mathbf{x}_0 = \mathbf{x}_0^Te^{\bar{A}^Tt_0}Ye^{\bar{A}t_0}\mathbf{x}_0 $ for any $Y$. Hence, $J_{eq}(Y) = J^*$.
\end{enumerate}
\end{proof}
\begin{lemma} \label{sec3_l2}
Consider the closed-loop system defined by \eqref{eq:controllaw} and \eqref{eq:singleagent}
with the cost function \eqref{eq:ihlqrcost}.
Assume that the matrix $A- \frac{1}{2}\Upsilon(I,P) 
$ is Hurwitz for some $P$. Then, the cost $J$ is finite and is given as $J= \mathbf{x}_0^T Y \mathbf{x}_0$ where $Y$ is the positive semidefinite solution of 
\begin{align}\label{sec3_l2_y_eq}
     A^TY+YA+Q  -\frac{1}{2}\Upsilon(P,Y)  + \frac{1}{4}\Upsilon(P,P) 
=\mathbf{0}.
\end{align}
\end{lemma}
\begin{proof}
Substituting 
\eqref{eq:controllaw} in \eqref{eq:singleagent}, it reads
$
\dot{\mathbf{x}} = \left( A- \frac{1}{2}\Upsilon(I,P)
\right) \mathbf{x},
$ and the cost as $
J = { \int_{t_0}^{\infty}  \mathbf{x}^T \left(Q+\frac{1}{4}\Upsilon(P,P) 
\right) \mathbf{x}}  dt
$.
Taking $\bar{A}=A- \frac{1}{2}\Upsilon(I,P)$, and $\bar{Q} = Q+\frac{1}{4}\Upsilon(P,P)$, then from Theorem 
\ref{sec3_thm3}, the cost is finite and is given as $J =\mathbf{x}_0^T Y \mathbf{x}_0$ where $Y$ satisfies \eqref{sec3_l2_y_eq}. 
\end{proof}
\label{sec3_olmi_slqr_sub_eslqr}
The following proposition synthesizes the $\epsilon-$suboptimal control for the IHLQR problem.
\begin{proposition}
\label{sec3_prop5}
Consider the system \eqref{eq:singleagent} with the cost \eqref{eq:ihlqrcost}. If the functional $s(\bullet)$ is strictly convex for matrices $P$ and $\bar{P}\succ 0$ satisfying \eqref{eq:bigLMI} 
with $A- \frac{1}{2}\Upsilon(I,P)$ being Hurwitz, then  
the following statements are true.
\begin{enumerate}
    \item The function $q=\mathbf{x}^T(P-\bar{P})\mathbf{x}$ is a solving function for the Problem \ref{sec2_p3}.
    \item The control law \eqref{eq:controllaw} is $\epsilon-$suboptimal 
    with $\epsilon=\mathbf{x}_0^T \bar{P} \mathbf{x}_0- J^* $.
\end{enumerate}
\end{proposition}
\begin{proof}
The inequality \eqref{eq:decomposedpbar} can be expressed as
$
    A^T\bar{P}+\bar{P}A+Q  +\frac{1}{4} \Upsilon(P,P) + \frac{1}{4} \Upsilon(\bar{P},\bar{P}) - \frac{1}{2} \Upsilon(P,\bar{P}) \prec \mathbf{0}.
    $
Since $\Upsilon(\bar{P},\bar{P})\succeq \mathbf{0}$, the following inequality holds true
\begin{align}\label{eq:decomposedpbarequiv}
    A^T\bar{P}+\bar{P}A+Q  +\frac{1}{4} \Upsilon(P,P) - \frac{1}{2} \Upsilon(P,\bar{P}) \prec \mathbf{0}.
\end{align}
Substituting for $Q+ \frac{1}{4} \Upsilon(P,P)$ in \eqref{eq:decomposedpbarequiv} using \eqref{sec3_l2_y_eq} and simplifying it reads as
$  \left( A^T- \frac{1}{2}\Upsilon(P,I)\right) \left(\bar{P}-Y \right)  + \left(\bar{P}-Y \right)\left( A- \frac{1}{2}\Upsilon(I,P)\right)  \prec \mathbf{0}.
$
Since the matrix $A-\frac{1}{2}\Upsilon(I,P)$ is Hurwitz, then using the Lyapunov stability theorem, it follows that $\bar{P}-Y \succ \mathbf{0}$, which implies that
$J= \mathbf{x}_0^T Y \mathbf{x}_0 < \mathbf{x}_0^T \bar{P} \mathbf{x}_0.
$
\end{proof}
It is worth noting that the $\epsilon-$suboptimality of the control law \eqref{eq:controllaw} in Proposition \ref{sec3_prop5} relies on the inequality \eqref{eq:decomposedpbar} for some $P$ and $\bar{P}\succ 0$. The next result provides the necessary and sufficient condition for their existence. 
\begin{theorem}\label{thm:stablesystems}
The inequality \eqref{eq:decomposedpbar} is satisfied for some $P$ and $\bar{P}\succ 0$ if and only if $A$ is Hurwitz.
\end{theorem}
\begin{proof}
Denoting $\check{Q} = Q + \frac{1}{4}
    \Upsilon(P+P^T-2\bar{P},P+P^T-2\bar{P}) $, the inequality \eqref{eq:decomposedpbar} can be equivalently written as
$    A^T\bar{P}+\bar{P}A+\check{Q}
    \prec \mathbf{0}.
    $
Note that $\check{Q}\succeq \mathbf{0}$. Since $\bar{P}\succ\mathbf{0}$, the matrix $A$ should be Hurwitz.
\end{proof}

As shown in Theorem \ref{thm:stablesystems}, the results stated in Proposition \ref{sec3_prop5} are valid only for stable systems. In the following, we shall generalize these results for unstable systems in Proposition \ref{sec3_prop6}.


\begin{lemma}
\label{sec3_l4}
Let $\bar{A} \in \mathbb{R}^{n \times n}$ be a Hurwitz matrix. Define $\tilde{P} \in \mathbb{R}^{n \times n}$ to be the solution of the equation
$
\bar{A}^T \tilde{P}+\tilde{P}\bar{A}+I =\mathbf{0}.
$
The following statements are true.
\begin{enumerate}
    \item There exists a matrix $P_{\eta}$ which satisfies
\begin{equation}
\label{sec3_l4_eq1}
\bar{A}^TP_{\eta}+P_{\eta}\bar{A} - \eta I \prec \mathbf{0},
\end{equation}
where $\eta \in \mathbb{R}$.
\item Furthermore, $P_{\eta}+\eta {\tilde{P}} \succ \mathbf{0}$.
\end{enumerate}
\end{lemma}
\begin{proof}
From Lemma \ref{sec3_l3}, the matrix $\tilde{P}$ always exists.
\begin{enumerate}
    \item The inequality \eqref{sec3_l4_eq1} can be written as $
\bar{A}^TP_{\eta}+P_{\eta}\bar{A}+V-\eta I = \mathbf{0}
$ for some matrix $V\succ \mathbf{0}$. 
The solution to the last equation, as per Lemma \ref{sec3_l3}, is then given as
$
P_{\eta}= \int_{0}^{\infty} \left( e^{t\bar{A}^T} \left( V- \eta I \right) e^{t\bar{A}} \right) dt.
$
\item The solution to the equation
$
\bar{A}^TP_v+P_v\bar{A}+V=\mathbf{0}
$
is given as
$
P_v = \int_{0}^{\infty} \left( e^{t\bar{A}^T} V e^{t\bar{A}} \right) dt,
$ and since $V \succ \mathbf{0}, P_v \succ \mathbf{0}$.
Consequently, the matrix $P_\eta$ can be expressed as $P_\eta = P_v - \eta \tilde{P}$.
Since $P_v \succ \mathbf{0}, P_\eta + \eta \tilde{P} \succ \mathbf{0}$.
\end{enumerate}
\end{proof}
\begin{proposition}
\label{sec3_prop6}
Consider the system \eqref{eq:singleagent} with the cost \eqref{eq:ihlqrcost}. If the functional $s(\bullet)$ is strictly convex for some matrix $P$, satisfying \eqref{sec3_l1_p_lmi} with $A- \frac{1}{2}\Upsilon(I,P)$ being Hurwitz, and some symmetric matrix $\bar{P}\succ 0$, whose solution is obtained by  solving the following convex optimization problem 
\begin{equation}\label{sec3_cop_eq1}
    \begin{gathered}
    \min\limits_{\bar{P}\succ\mathbf{0}} \eta \\
\text{ s.t. }
\begin{bmatrix}
-\Gamma(\bar{P}) + \eta I& \left( P+P^T-2\bar{P} \right) B \\
B^T\left(P+P^T -2\bar{P} \right)^T& 4R 
\end{bmatrix}
\succ \mathbf{0},
    \end{gathered}
\end{equation}
then  
the following statements are true.
\begin{enumerate}
    \item The function $q=\mathbf{x}^T(P-\bar{P})\mathbf{x}$ is a solving function for the Problem \ref{sec2_p3}.
    \item The control law \eqref{eq:controllaw} is $\epsilon-$suboptimal 
    with $\epsilon=  \mathbf{x}_0^T \left(\bar{P}+ \eta \tilde{P} \right) \mathbf{x}_0-J^* $,
    where $\tilde{P}$ is the solution of
$
\left( A^T- \frac{1}{2}\Upsilon(P,I) \right) \tilde{P}+ \tilde{P} \left( A- \frac{1}{2}\Upsilon(I,P) \right) +I = \mathbf{0}.
$
\end{enumerate}
\end{proposition}
\begin{proof}
Firstly, it is straightforward that for some symmetric $\bar{P}\succ \mathbf{0}$, \eqref{eq:decomposedpbar} implies \eqref{sec3_cop_eq1}. Subsequently, similar to the proof of Proposition \ref{sec3_prop5}, the inequality in \eqref{sec3_cop_eq1} 
can be simplified to 
$\left( A- \frac{1}{2}\Upsilon(P,I) \right)^T \left(\bar{P}-Y \right)
 +\left(\bar{P}-Y \right) \left( A- \frac{1}{2}\Upsilon(I,P) \right)- \eta I \prec \mathbf{0}.
 $
Taking $\bar{A} = \left( A- \frac{1}{2}\Upsilon(I,P) \right)$, and $P_\eta = \bar{P}-Y$, then from Lemma 
\ref{sec3_l4}, it follows that $\left( \bar{P}-Y + \eta \tilde{P} \right) \succ \mathbf{0}$, which implies that 
$
J= \mathbf{x}_0^T Y \mathbf{x}_0 < \mathbf{x}_0^T \left( \bar{P}+ \eta \tilde{P} \right) \mathbf{x}_0.
$
\end{proof}
The steps for computing the solution to the Problem \ref{sec2_p3} are summarized in the Algorithm \ref{alg:epsilonsuboptimal}. 
\begin{algorithm}
\SetAlgoLined
\caption{Computation of $\epsilon$-suboptimal control law}
\label{alg:epsilonsuboptimal}
\KwData{$A, B, Q \succeq \mathbf{0}, R \succ \mathbf{0}, \mathbf{x}_0$}
\KwResult{$P ,\epsilon$}
    Solve the inequality \eqref{sec3_l1_p_lmi} for some stabilizing matrix $P$, in the sense that the closed-loop system matrix $A-\frac{1}{2} \Upsilon(I,P)$ is Hurwitz.\\
  \eIf{$A$ is Hurwitz}{
   Compute the matrix $\bar{P}\succ\mathbf{0}$ by solving the inequality \eqref{eq:decomposedpbar} for the above stabilizing matrix\;
   Compute the control signal $\mathbf{u} = -\frac{1}{2}R^{-1}B^T \left( P+P^T \right) \mathbf{x}$, and $\epsilon=\mathbf{x}_0^T \bar{P} \mathbf{x}_0-J^*$.
   }{
   Compute the matrix $\bar{P}\succ\mathbf{0}$ and $\eta$ by solving the convex optimization problem \eqref{sec3_cop_eq1} for the above stabilizing matrix\;
   Solve $\left( A^T- \frac{1}{2}\Upsilon(P,I) \right) \tilde{P}+ \tilde{P} \left( A- \frac{1}{2}\Upsilon(I,P) \right) +I = \mathbf{0}$ for $\tilde{P}$\;
   Compute the control signal $\mathbf{u} = -\frac{1}{2}R^{-1}B^T \left( P+P^T \right) \mathbf{x}$, and $\epsilon=  \mathbf{x}_0^T \left(\bar{P}+ \eta \tilde{P} \right) \mathbf{x}_0-J^*$.
  }
\end{algorithm}

\section{Suboptimal Consensus Protocol Design 
}
\label{sec4_scpd}
In this section, the suboptimal solution to the consensus problem for multiagent systems is derived for agents communicating over undirected graph topologies.
The solution to the Problem \ref{sec2_scdp} is computed in two major steps
\begin{enumerate}
\item reformulate the multiagent dynamics \eqref{eq:agentdynamicsdiffB} as the standard regulation problem in terms of the neighbouring state error dynamics, and
\item obtain the (possibly) non-square \textit{structured} feedback gain matrix acting on the error signal by computing another matrix, which is square.
\end{enumerate}
Due to the structured feedback gain matrix, computing the solution to the underlying control optimization problem is an NP-hard problem \cite{borrelli2008dist}. Using the technique developed in the previous section, we shall compute the structured feedback gain matrix and the upper bound on the cost via solving a convex optimization problem.

Some terminologies related to graphs and Laplacian matrices are now briefly presented, which shall be used in later subsections \cite{fblns}.
\begin{definition}[connection topology]
A connection topology between the agents or a graph $\mathcal{G}$ is defined as $\mathcal{G} = (\mathcal{V,C})$, where $\mathcal{V}$ is the set of nodes (or vertices) $\mathcal{V}=\{1,2, \ldots,N\}$ and $\mathcal{C}\subseteq \mathcal{V}\times \mathcal{V}$ the set of edges $(i,j)$ with $i\in\mathcal{V},j\in\mathcal{V}$. 
\end{definition}
\begin{definition}[Laplacian matrix]
The Laplacian matrix of a graph $\mathcal{G}$ is defined as ${L}(\mathcal{G}) = {D}(\mathcal{G}) - {C}(\mathcal{G})$, where ${D}(\mathcal{G})$ is the diagonal matrix of vertex degrees, and ${C}(\mathcal{G})$ denotes the $(0,1)-$adjacency matrix of a graph $\mathcal{G}$.
\end{definition}
The class of $\mathcal{K}^N_{n,m}(\mathcal{G})$ matrices is defined in the following \cite{borrelli2008dist}.
\begin{definition}\label{def:classofK}
$\mathcal{K}^N_{n,m} = \{{M}\in\mathbb{R}^{nN\times mN}| {M}_{ij} = \mathbf{0} \text{ if } (i,j)\notin\mathcal{C}, {M}_{ij} = {M}[(i-1)n:in,(j-1)m:jm], i,j=1,\ldots,N \}$.
\end{definition}
\subsection{Error Dynamics}
\label{sec4_scpd_sub_ed}
Let denote the vectors $\hat{\mathbf{x}}\in\mathbb{R}^{nN}$ and $\hat{\mathbf{u}}\in\mathbb{R}^{mN}$, which collect the state and inputs of the $N$ systems \eqref{eq:agentdynamicsdiffB}, then the overall system can be compactly expressed as
\begin{subequations}\label{eq:multiagentX}
\begin{align}\label{eq:combMultiX}
    \dot{\hat{\mathbf{x}}} = \hat{{A}}\hat{\mathbf{x}} + \hat{{B}}\hat{\mathbf{u}}, \hat{\mathbf{x}}(t_0) = \hat{\mathbf{x}}_{t_0},
\end{align}
{where $\hat{{A}} = I_N\otimes {A}, \hat{{B}} = \mathtt{blkdiag}({B}_i),i=1,\ldots, N$ with $\otimes$ and $\mathtt{blkdiag}(\bullet)$ denoting the Kronecker product and block-diagonal matrices, respectively. Similarly, the corresponding cost functional \eqref{sec2_ocn_cost} can be equivalently
represented as}
\begin{align}\label{eq:combMultiJX}
\hat{J}=\int_{t_0}^\infty \left(\hat{\mathbf{x}}^T\hat{Q}\hat{\mathbf{x}}+\hat{\mathbf{u}}^T\hat{R}\hat{\mathbf{u}}\right)dt, 
\end{align}
{with $\hat{Q} = L\otimes \underline{Q} \in \mathcal{K}^N_{n,n}(\mathcal{G})$ and $\hat{R} = I_N\otimes \underline{R}$. As a consequence, Problem \ref{sec2_scdp} now reads as
to compute the control law of the form} 
\begin{align}\label{eq:combMultiUX}
\hat{\mathbf{u}}(t) = \hat{{K}}^{\hat{\mathbf{x}}}\hat{\mathbf{x}}(t), 
\end{align}
\end{subequations}
where $\hat{{K}}^{\hat{\mathbf{x}}}\in \mathcal{K}^N_{m,n}(\mathcal{G})$ and the upper bound $\gamma$ such that $\hat{J}<\gamma$. 
\begin{lemma}\label{lem:propertyKx}
The matrix $\hat{{K}}^{\hat{\mathbf{x}}}$ in \eqref{eq:combMultiUX} satisfies $\sum_{j=1}^{N}\hat{{K}}^{\hat{\mathbf{x}}}_{ij}=\mathbf{0}, i=1,\ldots,N.$
\end{lemma}
\begin{proof}
The above property can be readily shown by explicitly writing \eqref{sec2_ocn_ip} in matrix form for any connected graph.
\end{proof}
In \cite{jiao2019suboptimality}, the above problem is addressed by assuming that the upper bound $\gamma$ is \textit{a priori} given,  $\hat{{K}}^{\hat{\mathbf{x}}}_{ij} \forall i,j=1,\ldots,N$ are equal, and agents are homogeneous, i.e. $B_i=B, i =1,\ldots, N$ or equivalently $\tilde{B}=I_N\otimes B$. In this work, these assumptions are relaxed while synthesizing the feedback gain matrix even for homogeneous agents. In this context, the subsequent problem addressed in this paper is
\begin{problem}
\label{multi_ag_prob}
For the state dynamics \eqref{eq:combMultiX}, given a gain matrix $\hat{{K}}^{\hat{\mathbf{x}}}\in \mathcal{K}^N_{m,n}(\mathcal{G})$ such that the consensus is achieved, compute the upper bound $\gamma$ which satisfies $\hat{J}<\gamma$.
\end{problem}
It is worth noting that without considering the constraints 
on the feedback gain matrix $\hat{K}^{\hat{\mathbf{x}}}$ in \eqref{eq:combMultiUX}, the OCP in \eqref{eq:multiagentX} can be referred to as the centralized optimal control problem \cite{borrelli2008dist}. In that case, although the structure of the optimal control problem \eqref{eq:multiagentX} appears similar to that of Problem \ref{sec2_p3}, its solution as derived in the previous section cannot be used to solve the former problem because this solution relies on the fact that as $t_f\to\infty$ the vector $\mathbf{x}(t_f)=\mathbf{0}$, see (ESIHLQR), which may not hold in the case of \eqref{eq:combMultiX}-\eqref{eq:combMultiJX}.
To use the results derived for the single agent in the previous section to synthesize suboptimal controller for multiagent systems, we define a new state vector for the $i-$th agent as 
\begin{subequations}\label{eq:errorsignal}
\begin{align}\label{eq:newvector}
    \mathbf{e}_i = \mathbf{x}_i - \mathbf{x}_{i+1} , i=1,\ldots,N-1.
\end{align}
Then, the overall dynamics of a multiagent system can be equivalently represented by
\begin{align}\label{eq:maserror}
\dot{\mathbf{e}}= \tilde{A} \mathbf{e}+ \tilde{B} \hat{\mathbf{u}}, \mathbf{e}(t_0)= \mathbf{e}_0
\end{align}
where $\mathbf{e}=\mathtt{col}(\mathbf{e}_1, \mathbf{e}_2, \ldots, \mathbf{e}_{N-1}) = \mathtt{col}(\mathbf{x}_1 - \mathbf{x}_2, \mathbf{x}_2 - \mathbf{x}_3, \ldots, \mathbf{x}_{N-1} - \mathbf{x}_N) \in \mathbb{R}^{n(N-1)}, \hat{\mathbf{u}}= \mathtt{col}(\mathbf{u}_1, \mathbf{u}_2, \ldots, \mathbf{u}_N) \in \mathbb{R}^{mN}, \tilde{A}=I_{N-1}\otimes A$ and
\begin{align}\label{eq:btilde}
    \tilde{B} =
\begin{bmatrix}
B_1 & -B_2 & \cdots &\cdots&\cdots& \mathbf{0} \\
\mathbf{0} & B_2 & -B_3  & \cdots &\cdots & \mathbf{0}\\
\vdots &\vdots &\vdots  &\vdots &\vdots &\vdots \\
\mathbf{0}& \mathbf{0} & \cdots&\cdots& B_{N-1}&-B_{N}
\end{bmatrix}. 
\end{align}
\end{subequations}
\begin{theorem}\label{thm:x2e}
Let $\Delta_{nN}\in\mathbb{R}^{nN\times nN}$ is an upper triangular matrix  with all nonzero elements equal to $I_n$, then the following statements hold
\begin{enumerate}
\item $\Delta_{nN} = \Delta_{1N}\otimes I_n$.
    \item The vector $\hat{\mathbf{x}}$ can be expressed as \begin{align}\label{eq:transformation}
    \hat{\mathbf{x}} = \begin{bmatrix}
        \Delta_{n(N-1)} \\ \mathbf{0}
    \end{bmatrix}\mathbf{e} + (\mathbf{1}_N\otimes I_n)\mathbf{x}_N
\end{align}
where $\mathbf{1}_N$ is the $N-$dimensional vector with all entries equal to 1.
\item $\Delta_{nN}^T\hat{Q}\Delta_{nN} = \begin{bmatrix}
    \tilde{Q} & \mathbf{0}\\
    \mathbf{0} & \mathbf{0}
\end{bmatrix}$, where $\tilde{Q}\in\mathbb{R}^{n(N-1)\times n(N-1)}$.
\item The cost functional in \eqref{sec2_ocn_cost} can be compactly written as
\begin{align}\label{eq:costinE}
    \hat{J} = \int_{t_0}^\infty\left(\mathbf{e}^T\tilde{Q}\mathbf{e} + \hat{\mathbf{u}}^T\hat{R} \hat{\mathbf{u}}\right)dt.
\end{align}
\item The map from ${\mathbf{e}}$ to $\hat{\mathbf{u}}$, denoted by $\hat{K}^\mathbf{e}$, i.e. 
\begin{align}\label{eq:combmultiuine}
    \hat{\mathbf{u}} = \hat{K}^\mathbf{e} {\mathbf{e}},
\end{align}
with $\hat{K}^\mathbf{e} = \hat{K}^{\hat{\mathbf{x}}}\begin{bmatrix}
    \Delta_{n(N-1)}^T & \mathbf{0}^T
\end{bmatrix}^T$ is equivalent to \eqref{eq:combMultiUX}.
\end{enumerate}
\end{theorem}
\begin{proof}
\textit{1)} The equality can be shown by direct substitution.\\
    \textit{2)} Concatenating the two vectors on the right hand side (RHS) of \eqref{eq:transformation}, the distribution matrix of this concatenated vector becomes $\Delta_{nN}$, i.e. the RHS expression reads as $\Delta_{nN}\begin{bmatrix}
        \mathbf{e}^T & \mathbf{x}_N^T
    \end{bmatrix}^T$. Substituting \eqref{eq:newvector} in this RHS expression yields the left hand side (LHS) of \eqref{eq:transformation}. \\
    \textit{3)} Substituting $\hat{Q} = L\otimes \underline{Q}$ on the LHS, and using statement 1) and the property of product of two Kronecker products \cite{graham82}, the LHS can be written as $(\Delta_{1N}^T\otimes I_n)(L\otimes \underline{Q})(\Delta_{1N}\otimes I_n) = (\Delta_{1N}^T L\Delta_{1N})\otimes (I_n\underline{Q}I_n)$. Since every row sum and column sum of $L$ is zero, the last $n$ rows and columns of the LHS equal zero.\\
    \textit{4)} Since \eqref{eq:combMultiJX} is equivalent to \eqref{sec2_ocn_cost}, substituting  \eqref{eq:transformation} into \eqref{eq:combMultiJX} readily gives \eqref{eq:costinE}.\\
    \textit{5)} Substituting \eqref{eq:transformation} in \eqref{eq:combMultiUX} yields
    \begin{align}\label{eq:expanduwithx}
     \hat{\mathbf{u}} = \hat{K}^{\hat{\mathbf{x}}}\begin{bmatrix}
    \Delta_{n(N-1)}^T & \mathbf{0}^T
\end{bmatrix}^T\mathbf{e} + \hat{K}^{\hat{\mathbf{x}}}(\mathbf{1}_N\otimes I_n)\mathbf{x}_N   
    \end{align}
    From Lemma \ref{lem:propertyKx}, the second term in \eqref{eq:expanduwithx} equals zero.
\end{proof}
Thanks to Theorem \ref{thm:x2e}, Problem \ref{sec2_scdp} can  equivalently be stated as to determine the gain matrix $\hat{K}^\mathbf{e}$ and the upper bound $\gamma$ such that $\hat{J}<\gamma$.
\begin{remark}
\label{rem1}
It is worth noting that the gain matrix $\hat{K}^{\mathbf{e}}$ inherits the desired sparsity pattern of $\hat{K}^{\hat{\mathbf{x}}}$. Nevertheless, synthesizing the control law of the form \eqref{eq:combmultiuine} minimizing the functional \eqref{eq:costinE} subject to \eqref{eq:maserror} is an NP-hard problem because other than a sparsity pattern, which is readily available from the Definition \ref{def:classofK}, additional structural conditions on nonzero elements of the feedback gain matrix, i.e. $\hat{K}^{\hat{\mathbf{x}}}_{ij}, \forall (i,j)\in\mathcal{C}$,  due to the communication topology also need to be satisfied in order to implement the control law \eqref{sec2_ocn_ip}. \end{remark}
\subsection{Suboptimal Consensus Protocol Design}\label{se4_scpd_sub}

In this subsection, a suboptimal consensus protocol is synthesized using Krotov sufficient conditions and based on the results of the previous section. 
The feedback matrix of the closed-loop multiagent system defined by  \eqref{eq:maserror} and \eqref{eq:combmultiuine} is proposed as
\begin{equation}
\label{P_intro}
    \hat{K}^{\mathbf{e}}=-\hat{R}^{-1}\tilde{B}^T \hat{P},
\end{equation}
where $\hat{P} \in \mathbb{R}^{n(N-1) \times n(N-1)}$. It is worth noting that with the help of \eqref{P_intro}, the consensus protocol synthesis problem in terms of the matrix $\hat{K}^{\mathbf{e}}$ is now parametrized in terms of a square matrix $\hat{P}$. Since the matrix $\hat{K}^{\mathbf{e}}$ is required to satisfy the structural conditions as mentioned in Remark \ref{rem1}, the same needs to be translated onto the design matrix $\hat{P}$ such that for some structured feedback matrix $\hat{K}^{\mathbf{e}}$, \eqref{P_intro} is satisfied.

 
\begin{lemma}\label{lem:Pcond}
{\color{black}For any matrix $\hat{{K}}^{\hat{\mathbf{x}}}\in \mathcal{K}^N_{m,n}(\mathcal{G})$, 
\eqref{P_intro} admits a solution under the following conditions}
\begin{enumerate}
    \item $m \leq n$,
    \item $B_i, i=1,2,\ldots,N$ is a full rank matrix.
\end{enumerate}
\end{lemma}

\begin{proof}
Representing the block matrix $\hat{P}$ as
\begin{align*}
    &\hat{P}=\begin{bmatrix}
        \hat{P}_{1,1} & \hat{P}_{1,2}&\cdots& \hat{P}_{1,N-1}\\
        \hat{P}_{1,2}&\hat{P}_{2,2}&\cdots& \hat{P}_{2,N-1}\\
        \vdots&\vdots&\cdots&\vdots\\
        \hat{P}_{N-1,1} & \hat{P}_{N-1,2} &\cdots& \hat{P}_{N-1,N-1}
    \end{bmatrix} 
\end{align*} 
with $\hat{P}_{i,j} \in \mathbb{R}^{n \times n}$, $\hat{P}_{i,j}(\cdot,\cdot
)\in\mathbb{R}$ and similarly for $\hat{K}^{\mathbf{e}}_{i,j} \in \mathbb{R}^{m \times n}$ and $B_{i} \in \mathbb{R}^{n \times m}$, \eqref{P_intro} can be explicitly written as 
\begin{subequations}\label{eq:mappinglinear}
	\begin{gather}
		\left( \left(- \underline{R}^{-1}B_1^T \right) \otimes I_{n \times n} \right)  \begin{bmatrix}
			\hat{P}_{1,j}(1,1)\\\hat{P}_{1,j}(1,2)\\ \vdots\\ \hat{P}_{1,j}(n,n)
		\end{bmatrix} =  \begin{bmatrix}
			\hat{K}^{\mathbf{e}}_{1,j}(1,1)\\\hat{K}^{\mathbf{e}}_{1,j}(1,2)\\ \vdots\\ \hat{K}^{\mathbf{e}}_{1,j}(m,n)
		\end{bmatrix}, \label{eq:mappinglinear_a}\\
		\begin{multlined}
			\left( \left(- \underline{R}^{-1}B_i^T \right)\otimes I_{n \times n} \right) \begin{bmatrix}
				\hat{P}_{i,j}(1,1)-\hat{P}_{i-1,j}(1,1)\\
				\hat{P}_{i,j}(1,2)-\hat{P}_{i-1,j}(1,2)\\ \vdots\\ \hat{P}_{i,j}(n,n)-\hat{P}_{i-1,j}(n,n)
			\end{bmatrix} \\=  \begin{bmatrix}
				\hat{K}^{\mathbf{e}}_{i,j}(1,1)\\\hat{K}^{\mathbf{e}}_{i,j}(1,2)\\ \vdots\\ \hat{K}^{\mathbf{e}}_{i,j}(m,n)
			\end{bmatrix}; 
			\text{ for $1<i \leq N-1$}, \label{eq:mappinglinear_b}
		\end{multlined}\\
		\left( \left(-\underline{R}^{-1}B_N^T\right) \otimes I_{n \times n}  \right) \begin{bmatrix}
			\hat{P}_{N-1,j}(1,1)\\\hat{P}_{N-1,j}(1,2)\\ \vdots\\ \hat{P}_{N-1,j}(n,n)
		\end{bmatrix} =  \begin{bmatrix}
			\hat{K}^{\mathbf{e}}_{N,j}(1,1)\\\hat{K}^{\mathbf{e}}_{N,j}(1,2)\\ \vdots\\ \hat{K}^{\mathbf{e}}_{N,j}(m,n)
		\end{bmatrix},\label{eq:mappinglinear_c}
	\end{gather}
\end{subequations}
%
where $j=1,2,\ldots ,N-1$.

\textit{Case $1$: ($m<n$).}
 Since $B_i$ is a full rank matrix, it implies
 $\mathtt{rank}(B_i)=\mathtt{rank}(B_i^T)=\min(m,n)=m$. Also, $\mathtt{rank}(\underline{R}^{-1}B_i^T)=\mathtt{rank}(B_i^T)=m$. Note that in \eqref{eq:mappinglinear}, $\underline{R}^{-1}B_i^T \otimes I_{n \times n}\in\mathbb{R}^{mn\times nn}$ and
 using the rank property of the Kronecker product,
    $\mathtt{rank}(\underline{R}^{-1}B_i^T \otimes I_{n \times n})= mn$, which implies that 
    \eqref{eq:mappinglinear} admits infinite solutions.

    \textit{Case $2$: (m=n).} Under this case, $\left(\underline{R}^{-1}B_i^T\right) \otimes I_{n \times n} \in \mathbb{R}^{n^2 \times n^2}$ is an invertible matrix. Hence, \eqref{eq:mappinglinear_a} and \eqref{eq:mappinglinear_c} admit the unique solution, whereas \eqref{eq:mappinglinear_b} admits infinite solutions.
\end{proof}

\begin{corollary}\label{cor:linetopo}
For the line topology as shown in Fig. \ref{sec4_f1_lt}, $\hat{{K}}^{\hat{\mathbf{x}}}\in \mathcal{K}^N_{m,n}(\mathcal{G})$ for any block-diagonal matrix $\hat{P}$.
\begin{figure}[ht]
\includegraphics[width= \columnwidth]{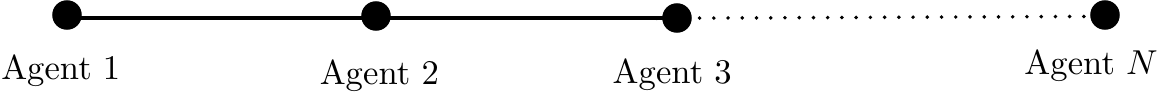}
\caption{Line Topology of agent interaction}
\label{sec4_f1_lt}
\end{figure}
\end{corollary}
\begin{proof}
For the line topology, the structure of the $\hat{K}^{\mathbf{e}}$ matrix is given as
\begin{equation*}
\hat{K}^{\mathbf{e}}=
 \begin{bmatrix}
 \bullet &\cdots & \cdots &\cdots & \mathbf{0}\\
 \bullet & \bullet& \cdots &\cdots&\mathbf{0}  \\
 \mathbf{0}& \bullet & \bullet & \cdots & \mathbf{0}\\
 \cdots &\cdots &\cdots &\cdots & \cdots \\
 \vdots &\vdots &\vdots &\vdots &\vdots   \\
 \mathbf{0}&\mathbf{0}&\cdots &\mathbf{0}& \bullet
 \end{bmatrix}.
\end{equation*}
From \eqref{eq:btilde} and \eqref{P_intro}, 
if $\hat{P}$ is a block-diagonal matrix then $\hat{K}^{\mathbf{e}}$ will have
the desired structure as above.
\end{proof}

One of the main consequences of Lemma \ref{lem:Pcond} is that the equation \eqref{eq:mappinglinear} reveals the structural requirements imposed by a network topology on $\hat{K}^\mathbf{e}$, which is fully satisfied by the block matrix elements of $\hat{P}$. 
Let $\mathbb{P}$ denote the set of all $\hat{P}$ matrices satisfying the network structural requirements on $\hat{K}^{\mathbf{e}}$ such that \eqref{P_intro} is satisfied.
Another main result of this paper is given in the next Proposition providing
sufficient conditions for addressing the suboptimal control design Problem \ref{sec2_scdp} in terms of the solution of a convex optimization problem.
 
\begin{proposition}
\label{sec4_thm7}
Consider the multiagent system \eqref{eq:maserror} with the cost \eqref{eq:costinE}. If the functional $s=\frac{\partial q}{\partial \mathbf{e}} \left( \tilde{A}\mathbf{e} + \tilde{B} \hat{\mathbf{u}} \right)+ \mathbf{e}^T \tilde{Q} \mathbf{e}+ \hat{\mathbf{u}} \hat{R} \hat{\mathbf{u}}$ with \eqref{eq:combmultiuine}-\eqref{P_intro} is strictly convex for a symmetric matrix $\hat{P}$, a positive-definite matrix $\hat{\bar{P}}$, whose solution is obtained by solving
the following convex optimization problem
\begin{align}
      \label{thm7cop1}
     \begin{gathered}
    \min\limits_{\hat{P},\hat{\bar{P}}\succ \mathbf{0}} \eta \\
    \text { s.t. } \hat{P} \in \mathbb{P}, \\
\left[\begin{array}{@{}c|c@{}}
  \begin{matrix}
 \bar{\Gamma}(\hat{{P}})& \hat{{P}}\tilde{B} \\
   \left( \hat{{P}}\tilde{B} \right)^T& \hat{R}
  \end{matrix}
  &  \mathbf{0} \\
\hline
  \mathbf{0} &
  \begin{matrix}
  - \bar{\Gamma}(\hat{\bar{P}}) +\eta I& \left(\hat{P} -\hat{\bar{P}} \right) \tilde{B} \\
   \left( \left(\hat{P} -\hat{\bar{P}} \right) \tilde{B} \right)^T& \hat{R}
  \end{matrix}
\end{array}\right]
\succ \mathbf{0},
    \end{gathered}
\end{align}
where $\bar{\Gamma}(\hat{P}) \triangleq \tilde{A}^T\hat{P}+\hat{P}\tilde{A}+\tilde{Q}$,
such that $\left(  \tilde{A}-\tilde{B} \hat{R}^{-1}\tilde{B}^T \hat{P} \right)$ is Hurwitz,
then the following statements are true.
\begin{enumerate}
\item The function $q=\mathbf{e}^T \left(\hat{P}- \hat{\bar{P}} \right) \mathbf{e}$ is a solving function for the problem.
\item The control input $\hat{\mathbf{u}}= \hat{K}^{\mathbf{e}}\mathbf{e}=-\hat{R}^{-1} \tilde{B}^T \hat{P} \mathbf{e}$ solves the Problem \ref{sec2_scdp} with $\gamma = \mathbf{e}_0^T\left( \hat{\bar{P}}+ \eta \tilde{P}_e \right) \mathbf{e}_0$, where $\tilde{P}_e$ is the unique positive semidefinite solution of the Lyapunov equation
\begin{equation}
   \left( \tilde{A}- \tilde{B} \hat{R}^{-1} \tilde{B}^T \hat{P} \right)^T \tilde{P}_e + \tilde{P}_e \left( \tilde{A}- \tilde{B} \hat{R}^{-1} \tilde{B}^T \hat{P} \right)=-I.
   \label{thm7_pe_eq}
\end{equation}
\end{enumerate}
\end{proposition}

\begin{proof}
Since $\hat{P} \in \mathbb{P}$, the matrix $\hat{K}^{\mathbf{e}}=-\hat{R}^{-1}\tilde{B}^T \hat{P}$ will have the desired sparsity pattern and structural conditions as per the topology requirements (see Remark \ref{rem1}). The second block of LMI can be reduced to
$\left( \tilde{A}- \tilde{B} \hat{R}^{-1} \tilde{B}^T \hat{P}\right)^T \left(
\hat{\bar{P}}-Y \right)
 +\left(
\hat{\bar{P}}-Y \right)\left( A- \tilde{B} \hat{R}^{-1} \tilde{B}^T \hat{P} \right)- \eta I \prec \mathbf{0}
$
similar to that in Proposition \ref{sec3_prop6}, where $Y$ is the weight matrix associated with the cost $\hat{J}=\mathbf{e}_0^T Y \mathbf{e}_0$. Since $\tilde{A}- \tilde{B} \hat{R}^{-1} \tilde{B}^T \hat{P}$ is Hurwitz, using Lemma \ref{sec3_l4}, we have
$
\hat{\bar{P}}-Y +\eta \tilde{P}_e \prec \mathbf{0} \implies \hat{J}= \mathbf{e}_0^T Y \mathbf{e}_0< \mathbf{e}_0^T\left( \hat{\bar{P}}+ \eta \tilde{P}_e \right) \mathbf{e}_0=\gamma.
$
\end{proof}

The results derived in this subsection shall now be used to provide a solution to Problem \ref{multi_ag_prob}, where a consensus protocol is given \textit{a priori}.
\begin{proposition}
\label{thm8}
For any $\hat{K}^{\hat{\mathbf{x}}} \in \mathcal{K}^N_{m,n}(\mathcal{G})$ which results in consensus, the upper bound $\gamma$ on the cost is $\gamma= \mathbf{e}_0^T\left( \hat{\bar{P}}+ \eta \tilde{P}_e \right) \mathbf{e}_0$ where $\hat{\bar{P}}$ is obtained by solving the following convex optimization problem
\begin{align}
      \label{thm8cop}
     \begin{gathered}
    \min \eta \\
    \text { s.t. } \tilde{B}^T \hat{P}=-\hat{R}\hat{K}^{\mathbf{e}}, \\
\left[\begin{array}{@{}c|c@{}}
  \begin{matrix}
 \bar{\Gamma}(\hat{{P}})& \hat{{P}}\tilde{B} \\
   \left( \hat{{P}}\tilde{B} \right)^T& \hat{R}
  \end{matrix}
  &  \mathbf{0} \\
\hline
  \mathbf{0} &
  \begin{matrix}
  - \bar{\Gamma} \left(\hat{\bar{P}} +\eta I \right)& \left(\hat{P} -\hat{\bar{P}} \right) \tilde{B} \\
   \left( \left(\hat{P} -\hat{\bar{P}} \right) \tilde{B} \right)^T& \hat{R}
  \end{matrix}
\end{array}\right]
\succ \mathbf{  0},
    \end{gathered}
\end{align}
and $\tilde{P}_e$ is obtained by solving
$
 \left( \tilde{A}- \tilde{B} \hat{R}^{-1} \tilde{B}^T \hat{P} \right)^T \tilde{P}_e + \tilde{P}_e \left( \tilde{A}- \tilde{B} \hat{R}^{-1} \tilde{B}^T \hat{P} \right)+I =\mathbf{0}.$
\end{proposition}


\color{black}
The steps for computing the suboptimal consensus protocol and the associated bound for any topology of agent-interaction are summarised in Algorithm \ref{alg2}. The steps for solving Problem \ref{multi_ag_prob} are summarised in Algorithm \ref{alg3}.

\begin{algorithm}
\SetAlgoLined
\caption{Computation of suboptimal consensus protocol}
\label{alg2}
\KwData{$A, B_i, \underline{Q} \succeq \mathbf{0}, \underline{R}  \succ \mathbf{0}, \mathbf{x}_{i0}$}
\KwResult{$\hat{K}^{\mathbf{e}} ,\gamma$}
    Compute the matrices $\tilde{A}$, $\tilde{B}$, 
 $\tilde{Q}$ and $\hat{R}$ following Theorem \ref{thm:x2e}.  \\
Obtain the conditions on the elements of the block matrix $\hat{P}$  
in accordance with the given agent-interaction topology
using \eqref{eq:mappinglinear}.\\
Solve the convex optimization problem \eqref{thm7cop1} in Proposition \ref{sec4_thm7} with the conditions on $\hat{P}$ as determined in Step $2$ above to obtain the triplet $\{\hat{P},\hat{\bar{P}} \succ \mathbf{0}, \eta \}$. \\
Compute the matrix $\hat{K}^{\mathbf{e}}=- \hat{R}^{-1}\tilde{B}^T \hat{P}$ and the corresponding bound $\gamma= \mathbf{e}_0^T \left(\hat{\bar{P}}+ \eta \tilde{P}_e \right) \mathbf{e}_0$
where $\tilde{P}_e$ is obtained by solving \eqref{thm7_pe_eq}.
\end{algorithm}
\begin{algorithm}
\SetAlgoLined
\caption{Computation of the upper bound on the cost for a given consensus protocol}
\label{alg3}
\KwData{$A, B_i, \underline{Q} \succeq \mathbf{0}, \underline{R}  \succ \mathbf{0}, \mathbf{x}_{i0}$, and a consensus protocol $\hat{K}^{\hat{\mathbf{x}}} \in \mathcal{K}^N_{m,n}(\mathcal{G})$.}
\KwResult{$\gamma$}
Compute 
$
    \hat{K}^\mathbf{e} = \hat{K}^{\hat{\mathbf{x}}}\begin{bmatrix}
    \Delta_{n(N-1)}^T & \mathbf{0}^T
\end{bmatrix}^T
$
as per Theorem \ref{thm:x2e}.\\
Solve the convex optimization problem \eqref{thm8cop} and the subsequent equation in Proposition \ref{thm8} to obtain $\gamma = \mathbf{e}_0^T\left( \hat{\bar{P}}+ \eta \tilde{P}_e \right) \mathbf{e}_0$ as the required bound satisfying $\hat{J}<\gamma$.\\
\end{algorithm}

\section{Numerical Results}
In this section, the application of  proposed techniques along with a  comparison with the existing results from the literature is demonstrated through numerical simulations. Firstly the usage of the results for suboptimal LQR design is demonstrated in subsection \ref{sec:suba} through two examples, where Example \ref{ex2} reports a comparison of the proposed approach with the existing result showcasing insights on the boundedness of the cost function for a stable scalar system. Subsequently, subsection \ref{sec:subb} presents numerical simulation results for multiagent systems under the different network topology.
\label{sec5_ne}
\subsection{Suboptimal LQR Design}\label{sec:suba}
\begin{example}[Open-loop Marginally Stable System]
\label{ex1}
Consider the system $\dot{x}=u$ with $x(0)=3$ and the cost $ { J =\int_{0}^{\infty}  (x^2 +u^2) dt }$. Compute $\epsilon$-suboptimal control law. For this example, the optimal cost is $J^*=9$.
\end{example}
\begin{sol*}
\normalfont

The inequalities, as per Proposition \ref{sec3_prop5}, which are needed to be solved for the parameters $p$ and $\bar{p}>0$ (these parameters are denoted in small letters for scalar systems), are given as
\begin{align*}
1-p^2 &>0  \\
-1-(p-\bar{p})^2 &>0.
\end{align*}
It is easily verified that these inequalities do not admit a solution, this is expected because the open-loop system is not stable, and hence convex optimization problem \eqref{sec3_cop_eq1}, as per Proposition \ref{sec3_prop6}, is solved. The solution of problem \eqref{sec3_cop_eq1} is computed to be
\begin{align*}
\eta=1, p=\bar{p}=0.64,  \tilde{p}=0.78.
\end{align*}
Correspondingly, $\epsilon= x_0^2(\bar{p}+\eta \tilde{p})-J^*=3.78$.
Thus, the control law $u=-0.64 x$ is $3.78$-suboptimal. The cost with this control input is computed to be $J=9.89$.
\end{sol*}

As mentioned in Section \ref{sec1_intro}, usually, an upper bound on the cost needs to be specified at the outset for using the existing results, which is based upon the following result from \cite{jiao2019suboptimality}.
\begin{theorem}
\label{sec5_thm8}
Consider the system \eqref{eq:singleagent} with cost functional \eqref{eq:ihlqrcost}.
Assume that the pair $(A,B)$ is stabilizable, and let $\epsilon >0$. Suppose that there exists a positive-definite $P$ satisfying 
\begin{align*}
 A^TP+PA- PBR^{-1} B^TP+Q &\prec \mathbf{0}      \\
 \mathbf{x}_0^T P \mathbf{x}_0 & < J^*+\epsilon.   \nonumber
\end{align*}
Then, the control input $\mathbf{u} \triangleq -R^{-1}B^TP \mathbf{x}$ results in a stable closed-loop system, i.e., $A-BR^{-1}B^TP$ is Hurwitz, and the corresponding cost $J$ is less than $J^*+\epsilon$.
\end{theorem}
\begin{table}[h!]
\caption{Range of parameters for Example \ref{ex2} with $\epsilon=1.27$}
\label{sec3_tab2}
\centering
\resizebox{\columnwidth}{!}{%
\begin{tabular}{ |c| c|c|}
\hline
\textbf{Conditions} & \textbf{Inequalities} & \textbf{Parameter-range} \\
\hline
Under study&
{$\!\begin{aligned} -2p+1-p^2 >0 \\
2\bar{p}-1-(p-\bar{p})^2 >0 \\
9\bar{p} < 5
\end{aligned}$}& {$\!\begin{aligned} p \in (-2.41, 0.41) \\
p \in (\bar{p}-\sqrt{2\bar{p}-1} , \bar{p}+\sqrt{2\bar{p}+1} ) \\
\bar{p} \in (0.50,0.56)
\end{aligned}$} \\
\hline
Existing &{$ \begin{aligned}
-2p+1-p^2 <0 \\
9p<5
\end{aligned}$} &{$\!\begin{aligned} p \in (0.41, \infty) \cup (-\infty, - 2.41) \\
{p} \in (-\infty,0.56)
\end{aligned}$}
\\
\hline
\multicolumn{3}{c}{\textbf{Parameter-range for stabilizing control law }}
\\
\hline
Under study &\multicolumn{2}{c|}{$p \in (0.22, 0.41)$} \\
\hline
Existing &\multicolumn{2}{c|}{$p \in (0.41, 0.56)$} \\
\hline
\end{tabular}
}
\end{table}

\begin{example}[Scalar Stable System]
\label{ex2}
Consider the system $\dot{x}=-x+u$; $x(0)=3$ with the cost $ { J =\int_{0}^{\infty} ( x^2 +u^2 ) dt }$. For this example, the optimal cost is $J^*=3.73$. For comparison purposes, let $\epsilon=1.27$.
\end{example}

\begin{figure}
\centering
\includegraphics[width=1\columnwidth]{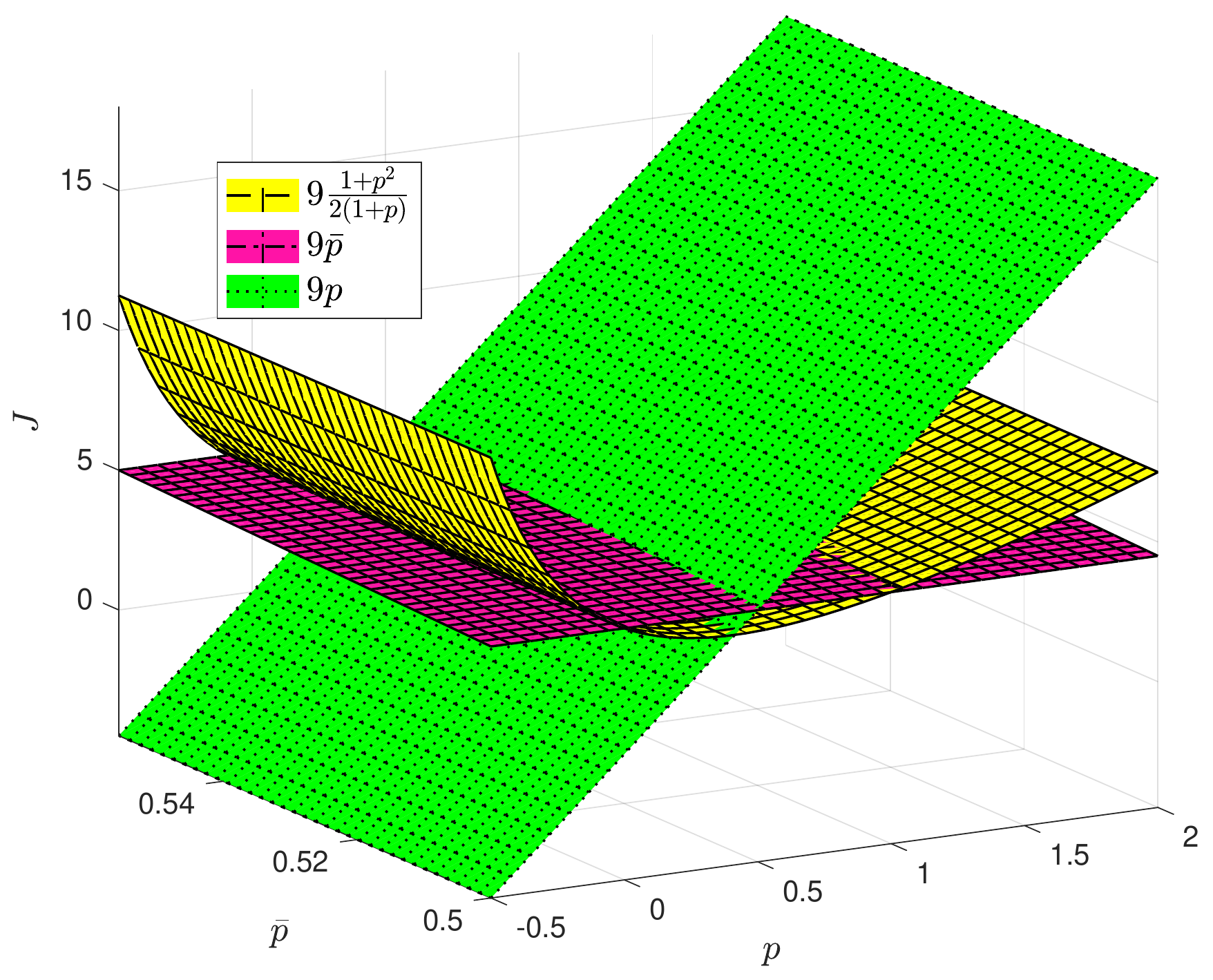}
\caption{Existing and new bounds for Example \ref{ex2}}
\label{sec3_f3}
\end{figure}
\begin{sol*}
\normalfont
The inequalities and the range of parameters $p$ and $\bar{p}$ for the conditions proposed and the existing conditions are detailed in Table \ref{sec3_tab2}. The plots of the quantities $px(0)^2$, $\bar{p}x(0)^2$ and the cost $J$ are shown in Figure \ref{sec3_f3}, where the function $px(0)^2$ is radially unbounded, which requires to specify the value of $\epsilon$ to use the existing results.
On the other hand, using the proposed approach, the bound is computed for any stabilising $p$. 
\end{sol*}
\subsection{Suboptimal Consensus Protocol Design}
\label{sec:subb}
The application of the derived conditions to the suboptimal consensus protocol design is demonstrated by examples in this subsection. An exposition with respect to the existing results is also presented. In example \ref{sec4_ex3}, suboptimal consensus protocols are designed for single-integrator agents communicating over the line topology (see Fig. \ref{sec4_f1_lt}) using the proposed method and state-of-the-art methods. Moreover, for a given consensus protocol, the degree of suboptimality is also computed. The demonstration of the proposed method is then extended to a network of second-order oscillators in example \ref{sec5_ex4}. While other network topologies are considered in examples \ref{sec5_ex5_ring} and \ref{sec4_ex6}.
Finally, application to a practical problem of consensus of rollers in the paper processing machine \cite{mosebach2013optimal} is demonstrated in example \ref{ex7}. All LMI-based problems formulated in this paper are solved using the CVX toolbox for MATLAB \cite{cvx}.

\begin{example}
\label{sec4_ex3}
Consider a four-agent system communicating over the line topology, where an agent is a single-integrator with 
dynamics
$
\dot{x}_i =u_i \ , i =  1,2,3,4
$
with $x_1(0)=0.1, x_2(0)= 0.2, x_3(0)= 0.5$ and $x_4(0)= -0.5$. Design a control law
such that consensus is achieved and compute the upper bound on the cost \eqref{sec2_ocn_cost} with $\underline{Q}=\underline{R}=I$.
\end{example}
\begin{sol*}
\normalfont
For the line topology $\tilde{A}=\mathbf{0}$, $\tilde{B}=\begin{bmatrix}
   1 & -1 &0 &0\\
   0 & 1 &-1 & 0\\
   0 &0& 1 & -1
\end{bmatrix}$ and $\tilde{Q}=\hat{R}=I_{3 \times 3}$.
\begin{enumerate}[(a)]
 

\item\textit{Solution using the Algorithm \ref{alg2}:}

The solution of the convex optimization problem \eqref{thm7cop1} is computed to be $\eta=1$ with $\hat{P}=\hat{\bar{P}}=\mathtt{diag}(0.39,0.37,0.39)$ (see Corollary \ref{cor:linetopo}),  where $\mathtt{diag}(\bullet)$ denotes the diagonal matrix. The corresponding feedback-gain matrix is given as
\begin{equation*}
\hat{K}^{\mathbf{e}}= \begin{bmatrix}
-0.39 &0 & 0\\
0.39 &-0.37 & 0\\
0 &0.37 &-0.39 \\
0 & 0 & 0.39\\
\end{bmatrix},
\end{equation*}
with
$\gamma = \mathbf{e}_0^T \left(\hat{\bar{P}}+ \eta \tilde{P}_e \right) \mathbf{e}_0=1.11 $. Thus, the consensus protocol $\hat{\mathbf{u}} = \hat{K}^\mathbf{e} {\mathbf{e}}$ is $1.11-$suboptimal. 
The cost with this control law is computed to be $\hat{J}=0.89$.
\item {\textit{Comparison with existing results in literature}:}

The design technique proposed in \cite{jiao2019suboptimality} is used for the comparison. As discussed earlier, the upper bound in their technique is \textit{a priori} specified based on which a set of initial conditions of agents is determined. Also, the consensus protocol is designed separately without using the knowledge of the upper bound.  
The eigenvalues of the Laplacian matrix are $ 0.58,2,3.41$. 
Subsequently, the following equation is solved for computing the design parameter $p>0$.
\begin{equation*}
\left(c^2(3.41)^2-2c(3.41)\right)p^2+3.41+0.001=0,
\end{equation*}
with $c= \dfrac{2}{(0.58+3.41)}=0.5$.
It gives $p=2.61$, thus yielding the feedback-gain matrix as 
\begin{equation*}
\hat{K}^{\mathbf{e}}= \begin{bmatrix}
-cp &0 & 0\\
cp &-cp & 0\\
0 &cp &-cp \\
0 & 0 & cp\\
\end{bmatrix},
\end{equation*}
where $cp=1.31$. Assuming $\gamma = 1.11$ as computed above, the initial condition of the multiagent system should satisfy 
$\|\hat{\mathbf{x}}_0\|<\sqrt{1.11/p}=0.65$. However, $\|\hat{\mathbf{x}}_0\| = 0.74$. This implies that the upper bound cannot be selected arbitrarily  in their design procedure. If the upper bound is taken as $\gamma = 1.5$, which is a loose bound in comparison to the above, then $\|\hat{\mathbf{x}}_0\|<\sqrt{1.5/p}=0.76$, which implies $\hat{J}<1.5$. The cost with this feedback gain is computed to be $\hat{J}=0.92$. It is worth noting that the actual cost still satisfies $\hat{J}<1.11$ with the given initial condition of agents.

\item \textit{Solution using the Algorithm \ref{alg3}:
}

Let $\hat{K}^{\mathbf{\hat{x}}}=-\alpha L$, where $\alpha$ is the design parameter, and for any $\alpha>0$, the multiagent system achieves consensus \cite{mesbahi2010graph}.
Using Theorem \ref{thm:x2e}, the matrix $\hat{K}^{\mathbf{e}}$ is expressed as 
\begin{equation*}
    \hat{K}^{\mathbf{e}}=  \alpha\begin{bmatrix}
        -1 &0 & 0\\
        1&-1&0\\
        0& 1&-1\\
        0&0& 1
    \end{bmatrix}.
\end{equation*}
For different values of $\alpha$, the upper bound and the actual cost are listed in Table 
\ref{sec5_b_thm8_res}.
\begin{table}[h!]
\caption{Upper-bound and cost for different values of $\alpha$ in Example \ref{sec4_ex3}}
\label{sec5_b_thm8_res}
\centering
\resizebox{0.6\columnwidth}{!}{%
\begin{tabular}{ |c| c|c|}
\hline
$\alpha$&  Upper bound, $\gamma$ & Cost, $\hat{J}$ \\
\hline
$0.1$& $2.75$&$2.70$
\\
\hline
0.2 &$1.54$&$1.43$\\
\hline
0.3 &$1.21$&$1.04$ \\
\hline
$0.4$&$1.09$ &$0.88$ \\
\hline
0.5 &$1.07$&$0.80$ \\
\hline
\end{tabular}
}
\end{table}
\end{enumerate}

\end{sol*}

\begin{example}
\label{sec5_ex4}
Consider a four-agent system where second-order oscillators communicate over the line topology. The dynamics of each agent is described by
$
    \dot{\mathbf{x}}_i=A \mathbf{x}_i+ \mathbf{b} \mathbf{u}_i, i = 1,2,3,4,
$
with 
$A=\begin{bmatrix}
0 & 1 \\
-1 & 0
\end{bmatrix}$, 
$\mathbf{b}= \mathtt{col}(0,1)
$. Also, $\mathbf{x}_1(0)= \mathtt{col}(0.35,0.15),
\mathbf{x}_2(0)= \mathtt{col}(0.26,0.48),
\mathbf{x}_3(0)= \mathtt{col}(-0.24,-0.22) 
$ 
and $\mathbf{x}_4(0)= \mathtt{col}(-0.30,-0.12) 
$. Compute a control law 
such that consensus is achieved and compute the upper bound on the cost \eqref{sec2_ocn_cost}  with $\underline{Q}=\underline{R}=I$.
 \end{example}
\begin{sol*}
\normalfont
For this case $\tilde{Q}=I_{6 \times 6}$. 
For symmetric and block-diagonal conditions on $\hat{P}$, the solution of convex optimization problem \eqref{thm7cop1} is computed to be
$\eta=1$,
\begin{align*}
\hat{P}&= \mathtt{blkdiag}(\hat{P}_{11}, \hat{P}_{22},\hat{P}_{33}),\\
\hat{\bar{P}}&=\mathtt{diag}(0.44,0.44,0.41,0.41,0.44,0.44),
\end{align*}
where $\hat{P}_{11}=\mathtt{diag}(0.43,0.44), \hat{P}_{22}=\mathtt{diag}(0.40, 0.41), \hat{P}_{33}=\mathtt{diag}(0.43,0.44)$. Subsequently, the feedback-gain matrix, upper bound, and the actual cost are computed as
\begin{align*}
    \hat{K}^{\mathbf{e}}= \begin{bmatrix}
0 &  -0.44         &0        & 0       &  0 &        0\\
   0&    0.44  &0   &-0.41 &      0 &        0 \\
         0 &        0   &0   & 0.41 &0  & -0.44 \\
         0        & 0  &       0     &    0&    0&    0.44
\end{bmatrix},
\end{align*}
$\gamma=2.26$ and $J=2.08$, respectively.
\end{sol*}

\begin{figure}
     \centering
     \begin{subfigure}[c]{0.45\columnwidth}
         \centering
         \includegraphics[width=\textwidth]{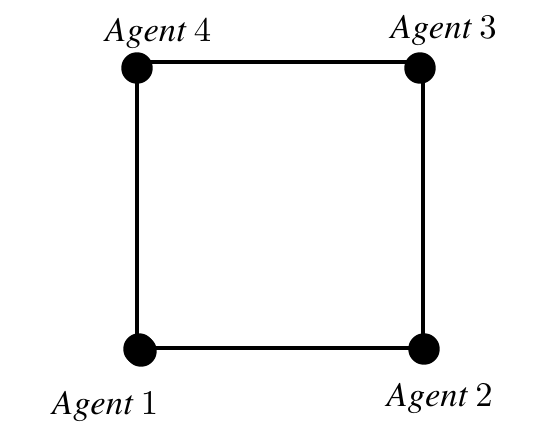}
         \caption{Ring topology}
         \label{sec4_ring}
     \end{subfigure}
     \hfill
     \begin{subfigure}[c]{0.5\columnwidth}
         \centering
         \includegraphics[width=\textwidth]{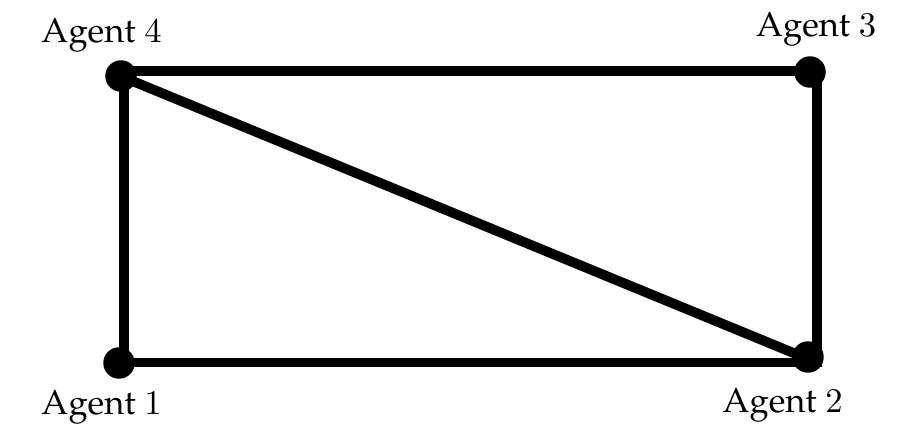}
         \caption{Topology for Example \ref{sec4_ex6}}
         \label{sec5_acomp}
     \end{subfigure}
        \caption{Network topology}
        \label{fig:topology}
\end{figure}

 \begin{example}
\label{sec5_ex5_ring}
Consider a four-agent system with single-integrator agents communicating over the ring topology, as shown in Figure \ref{sec4_ring}. The dynamics and initial condition of each agent are the same as those in example \ref{sec4_ex3}.
Let $\underline{Q}=\underline{R}=I$. Design a suboptimal consensus protocol and the degree of suboptimality. 
\end{example}

\begin{sol*}
\normalfont

To ensure that $\hat{K}^{\mathbf{e}}$ satisfies the network imposed structural requirements, the conditions on the matrix $\hat{P}$ are determined using \eqref{P_intro} as
$
\hat{P}(1,2)=\hat{P}(1,3)=\hat{P}(2,3).
$
Also, for this case $\tilde{Q} =\begin{bmatrix}
2& 1&1 \\
1&2 &1 \\
1& 1& 2
\end{bmatrix}.$
With these conditions, solving the convex optimization problem \eqref{thm7cop1} yields
$\eta=4, \hat{P}=
\begin{bmatrix}
    0.59   & 0.43&    0.43 \\
    0.43   & 0.70 &   0.43 \\
    0.43 &   0.43   & 0.59
\end{bmatrix}
$
and
$
\hat{\bar{P}}=\begin{bmatrix}
     1.01 &   0.27 &   0.16 \\
    0.27  &  1.02 &   0.27 \\
    0.16  &  0.27 &   1.01
\end{bmatrix}
$.
Correspondingly, 
the feedback-gain matrix is
$
\hat{K}^{\mathbf{e}}= \begin{bmatrix}
     -0.59 &  -0.43  & -0.43 \\ 
    0.15   &-0.27   &      0 \\
         0   & 0.27  & -0.16 \\
    0.43   & 0.43 &   0.59
\end{bmatrix}.
$
 Also, $\mathbf{e}_0^T \left({\hat{\bar{P}}}+ \eta \tilde{P}_e \right) \mathbf{e}_0=5.69$ with the actual cost computed to be $\hat{J}=1.52$.
\end{sol*}
\begin{example}
\label{sec4_ex6}
Consider a four-agent system with dynamics and initial condition of each agent similar to that in example \ref{sec4_ex3} communicating over the topology as shown in Figure \ref{sec5_acomp}.
Let $\underline{Q}=\underline{R}=I$. Design a suboptimal consensus protocol and the degree of suboptimality.
\end{example}
\begin{sol*}
\normalfont
Using \eqref{P_intro}, the following condition is imposed on the matrix $\hat{P}$ to ensure that the feedback-matrix is of the desired structure
$
\hat{P}(1,2)=\hat{P}(1,3)=\hat{P}(3,1)=\hat{P}(2,1).
$ Also, 
$
\tilde{Q}=\begin{bmatrix}
    2 &1 &1 \\
    1 & 3 & 2 \\
    1 & 2 & 3
\end{bmatrix}.
$
With these conditions, the solution to the convex optimization problem \eqref{thm7cop1} is computed to be
$\eta=5.56$ with $\hat{P}=
\begin{bmatrix}
    0.53  &  0.42 &   0.42 \\
    0.42   & 0.92 &   0.69 \\
    0.42  &  0.69  &  0.82
\end{bmatrix}
$
and
$
\hat{\bar{P}}=
\begin{bmatrix}
    1.19 &   0.31 &   0.15 \\
    0.32  &  1.28   & 0.39 \\
    0.16   & 0.39   & 1.27
\end{bmatrix}.
$ Correspondingly,
$\hat{K}^{\mathbf{e}}= \begin{bmatrix}
     -0.53 &  -0.42 &  -0.42 \\
    0.11  & -0.50 &  -0.27 \\
         0   & 0.22  & -0.12\\
    0.42  &  0.69&    0.82
\end{bmatrix}
, \gamma=\mathbf{e}_0^T \left({\hat{\bar{P}}}+ \eta \tilde{P}_e \right) \mathbf{e}_0=6.72
$ and the cost is computed to be $J=1.61$.
\end{sol*}

\begin{example}{(Roller Consensus Problem)}
\label{ex7}
\begin{figure}
    \centering
    \includegraphics[width=0.7\columnwidth]{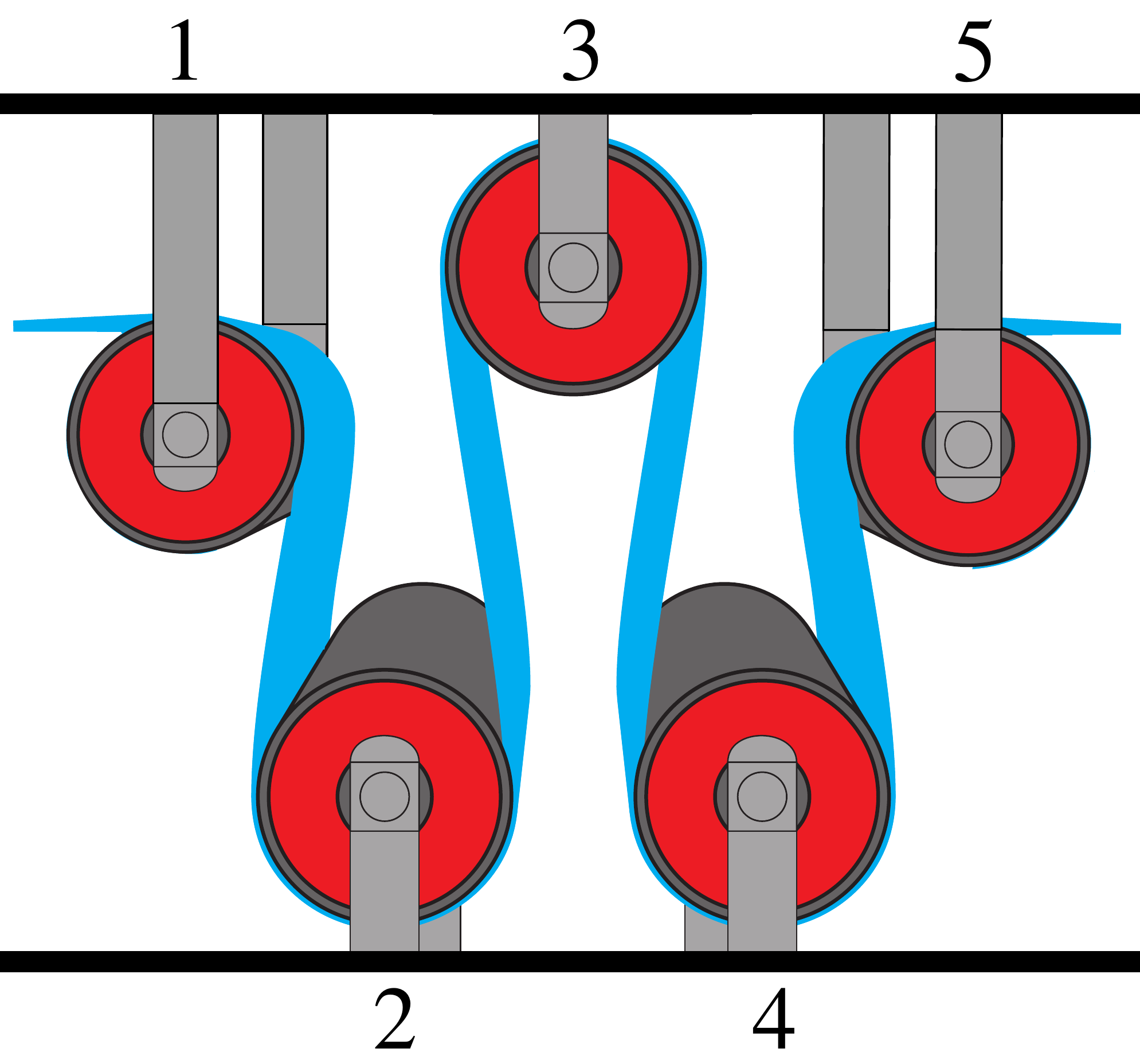}
    \caption{Drying section of a paper processing machine (adopted with modifications from \cite{mosebach2013optimal})}
    \label{fig:papermachine}
\end{figure}
Consider the drying section of the paper processing machine as shown in Figure \ref{fig:papermachine} consisting of five rollers, where roller $1$ exchanges information with roller $2$, roller $2$ with roller $1$ and roller $3$ and so on.  
The dynamics of each roller is governed by  
$
\dot{\mathbf{x}}_i= A \mathbf{x}_i+ \mathbf{b}_i \mathbf{u}_i
$, where $A = \begin{bmatrix}
0 & 1 & 0\\
0 & -0.01 & 0.2\\
0 & 0& -125
\end{bmatrix}, b_1= \mathtt{col}(0,0,20),
b_2= \mathtt{col}(0,0,18), 
b_3= \mathtt{col}(0,0,16),
b_4= \mathtt{col}(0,0,14), 
b_5=\mathtt{col}(0,0,12)
$
and initial condition as 
$
\mathbf{x}_1(0)= \begin{bmatrix}
   0.02 &
         0&
    0.01
\end{bmatrix}^T$, $ 
\mathbf{x}_2(0) =\begin{bmatrix}
           0.01&
    0.01&
   -0.01
\end{bmatrix}^T$, $
\mathbf{x}_{3}(0)=\begin{bmatrix}
     0.05 &
    0.01&
    0.01
\end{bmatrix}^T$, $
\mathbf{x}_4(0)=\begin{bmatrix}
        0.04&
    0.02&
    0.02
\end{bmatrix}^T$ \text{ and }
$\mathbf{x}_5(0)=\begin{bmatrix}
  0.07&   0&    0
\end{bmatrix}^T$.
The objective is to compute the diffusive control law such that all  states reach an agreement as $t \rightarrow \infty$ and minimize the cost functional 
$
\hat{J}=\int_{0}^{\infty}10\sum_{i =1}^{5}\sum_{j \in \mathcal{N}_i} 
\left(\mathbf{x}_i-\mathbf{x}_j \right)^T  \left(\mathbf{x}_i - \mathbf{x}_j \right)
+ 0.001 \sum_{i=1}^{5}\mathbf{u}_i^T \mathbf{u}_i  dt.
$
Also, compute the upper bound on $\hat{J}$.
\label{roller_prob}
\end{example}
\begin{sol*}
\normalfont
\begin{figure}
    \centering
    \includegraphics[width=0.7\columnwidth]{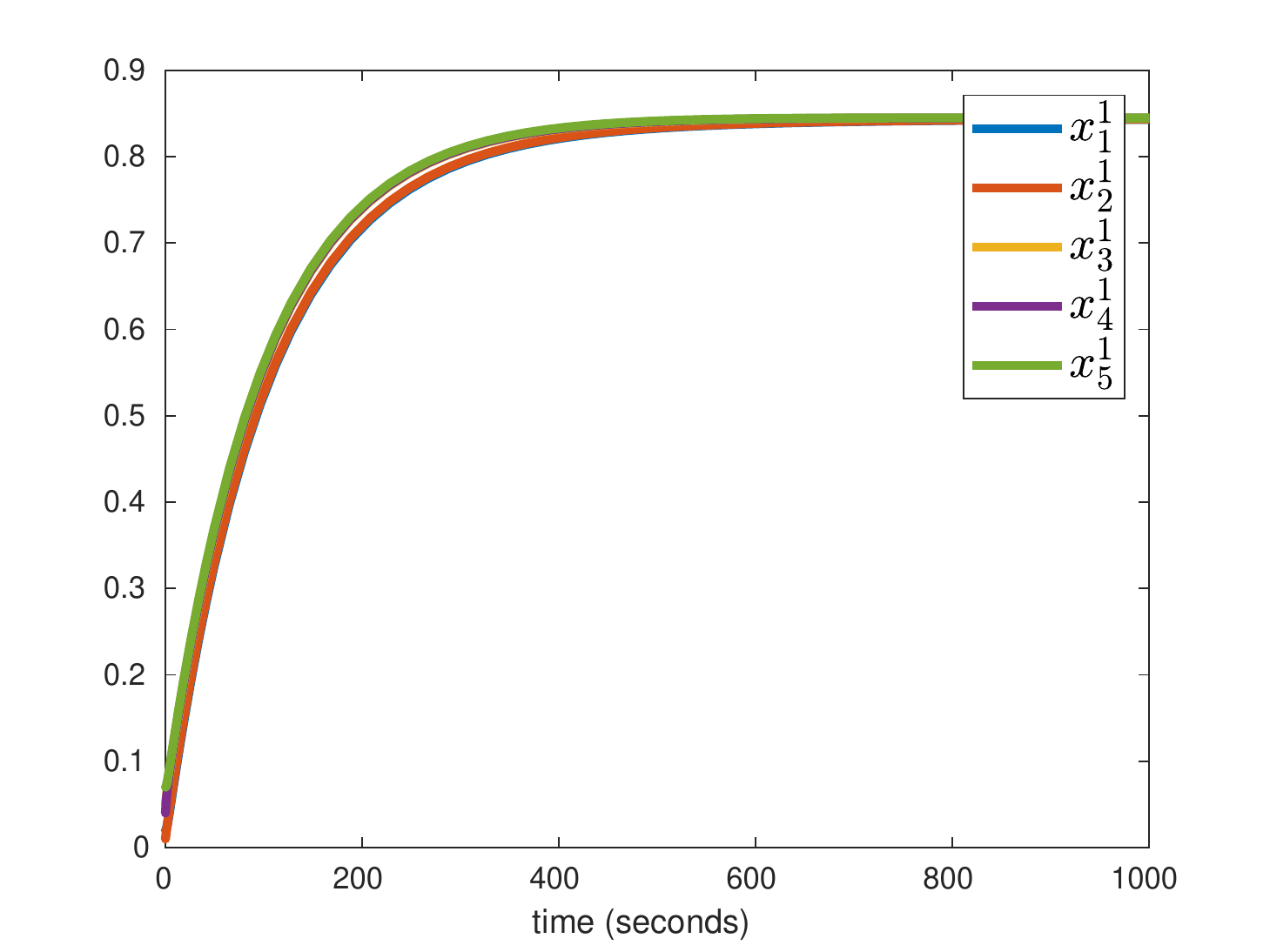}
    \caption{Consensus of first state of agents in Example \ref{roller_prob}}
    \label{fig:roller_s1}     
\end{figure}
For this problem $\tilde{Q}=10I_{12 \times 12}$. A solution to the convex optimization problem \eqref{thm7cop1} with symmetry and block-diagonal conditions on $\hat{P}$ matrix is obtained with $\eta=10$. Subsequently, the feedback-gain matrix is computed as 
given in \eqref{K_roller}. Correspondingly, $\gamma=2.49$ and cost $J=2.48$.
The closed-loop state trajectories are shown in Figures \ref{fig:roller_s1}, \ref{fig:roller_s2} and \ref{fig:roller_s3}.


\begin{figure*}[ht]
	{\small
		\begin{align}\label{K_roller}
			\hat{K}^{\mathbf{e}}= \begin{bmatrix}
				-0.48&-131.40&-32.69&0&0&0&0&0&0&0&0&0\\
				0.43&118.26&29.42&-0.43&-118.83&-28.17&0&0&0&0&0&0\\
				0&0&0&0.38&105.62&25.047&-0.44&-121.09&-28.20&0&0&0\\
				0&0&0&0&0&0&0.38&105.96&24.68&-0.51&-135.99&-30.81\\
				0&0&0&0&0&0&0&0&0&0.44&116.57&26.40\\
			\end{bmatrix}
		\end{align}
	}
\end{figure*}

\begin{figure}
    \centering
    \includegraphics[width=0.7\columnwidth]{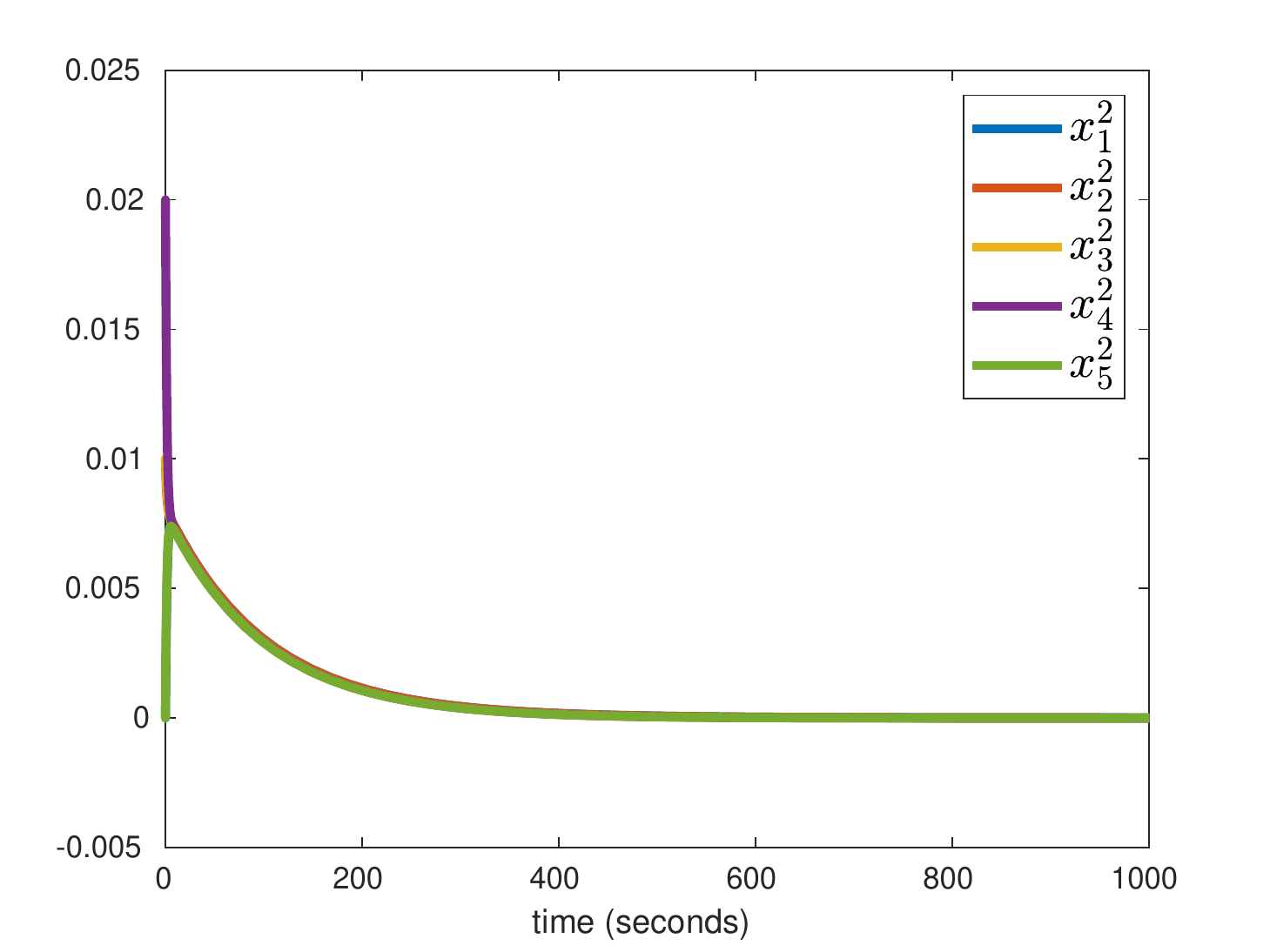}
    \caption{Consensus of second state of agents in Example \ref{roller_prob}}
    \label{fig:roller_s2}
\end{figure}

\begin{figure}
    \centering
    \includegraphics[width=0.7\columnwidth]{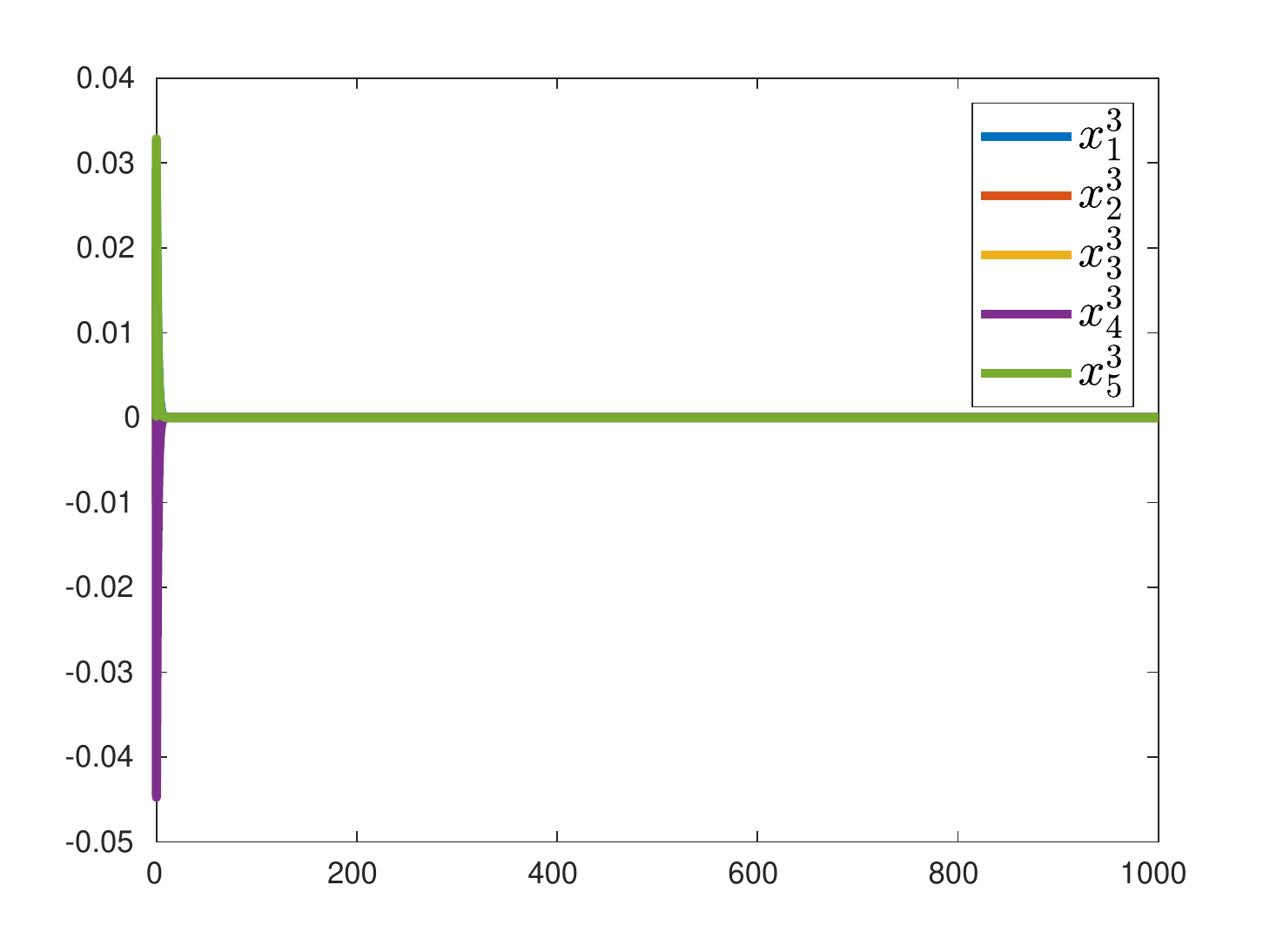}
    \caption{Consensus of third state of agents in Example \ref{roller_prob}}
    \label{fig:roller_s3}
\end{figure}

\end{sol*}
\section{Conclusion}
\label{sec6_conc}
This work presents a new design method for the suboptimal consensus protocol for   
a class of multiagent systems. The design method is derived using the Krotov sufficient conditions for global optimality. One of the striking advantages of this design method is that it also computes the upper bound on the cost functional compared to existing methods in the literature, where it is \textit{a priori} given and may render an infeasible solution. As a consequence, the degree of suboptimality with respect to the linear-quadratic cost functional is also determined for any consensus protocol. It is well-known that although the optimal control problems formulated in Krotov framework may be easier to solve, the selection of the solving functions is not trivial. Consequently, iterative algorithms are typically used to solve these problems. With the proposed solving functions, a direct solution to the above problems is obtained. Moreover, while designing suboptimal consensus protocols, we have also addressed the suboptimal LQR problem for linear systems using the Krotov sufficient conditions. A new algorithm is proposed for computing the feedback-gain matrix. This algorithm gives new insights facilitating the design of suboptimal control laws. The former problem  
is recast as an error-regulation problem, and the network-imposed structural requirements on the control input are translated on another design matrix in the appropriately formulated convex optimization problem. The solution of this convex optimization problem is then used to compute the consensus protocol and bound on the corresponding cost. The proposed approach does not assume identical gain matrices even for homogeneous agents. Several numerical examples are considered to  demonstrate the significance and usability of the proposed techniques.

Future work includes analysing the issue of computing asymmetric matrices in the formulated convex optimization problem. Furthermore, the objective function in the formulated convex optimization problem can be chosen to meet other design specifications (for example, minimization of the norm of control input) by exploiting the matrix variables available in the set defined by the linear matrix inequalities.

\bibliographystyle{ieeetr}
\bibliography{kro_ref}

\begin{IEEEbiography}
[{\includegraphics[width=1in,height=1.25in,clip,keepaspectratio]{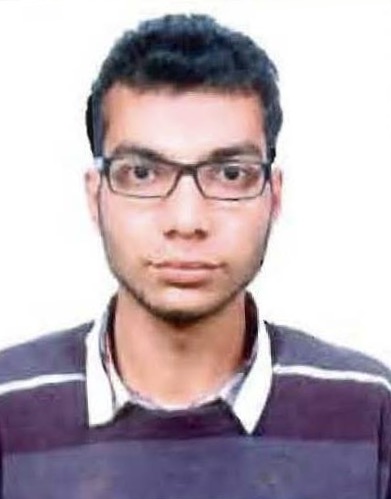}}]
{Avinash Kumar} was born in Hamirpur, Himachal Pradesh, India in 
1993. He received the B.Tech. degree in electronics and communication engineering from Rayat Institute of Engineering and Information Technology (RIEIT), Ropar affiliated to Punjab Technical University, Jalandhar, India in $2014$. He received the M.Tech. degree in electrical engineering with specialisation in signal processing and control from national institute of technology, Hamirpur, Himachal Pradesh, India in $2016$. Currently, he is a PhD scholar in the school of computing and electrical engineering (SCEE) at Indian Institute of Technology (IIT), Mandi, Himachal Pradesh, India. His research interests include optimal control, linear matrix inequalities, convex optimization, nonlinear systems and multi-agent systems.
\end{IEEEbiography}

\begin{IEEEbiography}
[{\includegraphics[width=1in,height=1.25in,clip,keepaspectratio]{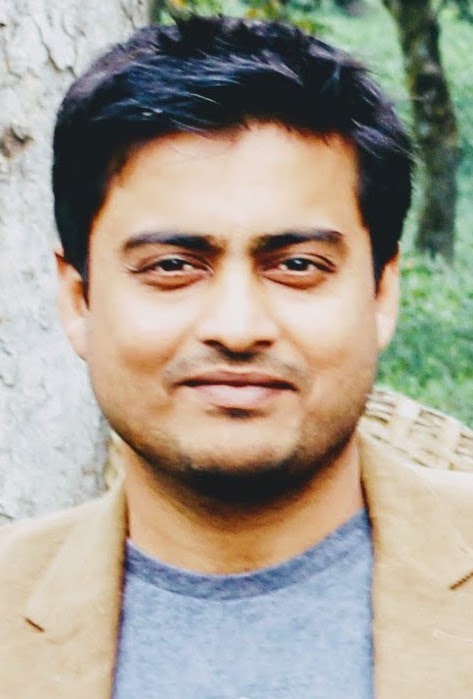}}]{Tushar Jain} received the degree of Doctor in Control, Identification and Diagnostic from Université de Lorraine, Nancy, France in 2012. He previously received the degree of M.Tech. in System modeling and control from Indian Institute of Technology (IIT) Roorkee in 2009. From 2013 to 2014 and 2014 to 2015, he was a Post-doc researcher and Academy of Finland researcher respectively in the Research Group of Process Control at Aalto University, Finland. Since 2015, he is with the School of Computing and Electrical Engineering, IIT Mandi. During these last years, his research interest has been mainly concentrated on mathematical control theory, fault-tolerant control, fault diagnosis with applications to renewable energy systems. He has received thrice the best paper award for his research work. He has authored a book entitled Active Fault-Tolerant Control Systems: A Behavioral System Theoretic Perspective.
\end{IEEEbiography}

\end{document}